\documentclass[psamsfonts]{amsart}

\usepackage{amssymb,amsfonts,amsthm}
\usepackage[all,arc]{xy}
\usepackage{enumerate}
\usepackage{mathrsfs}
\usepackage{graphicx}
\usepackage{tikzexternal}

\usepackage{hyperref}
\usepackage{mathdots}
\usepackage{cleveref}
\usepackage{color}
\usepackage{bm}
\usepackage{hyperref}
\usepackage{array}
\usepackage{mathrsfs}

\newtheorem{thm}{Theorem}[section]
\newtheorem{cor}[thm]{Corollary}
\newtheorem{prop}[thm]{Proposition}
\newtheorem{lem}[thm]{Lemma}

\newtheorem*{thm1}{Theorem}

\theoremstyle{definition}
\newtheorem{defn}[thm]{Definition}

\newtheorem{exmp}[thm]{Example}
\newtheorem{exmps}[thm]{Examples}

\theoremstyle{remark}
\newtheorem{rem}[thm]{Remark}

\newcommand{\Hom}{\operatorname{Hom}}
\newcommand{\Aut}{\operatorname{Aut}}
\newcommand{\R}{\mathbb{R}}
\newcommand{\Z}{\mathbb{Z}}
\newcommand{\C}{\mathbb{C}}

\newcommand{\G}{\mathfrak{G}}
\newcommand{\W}{\mathfrak{W}}
\newcommand{\cG}{\mathcal{G}}
\newcommand{\GH}{\mathfrak{GH}}
\newcommand{\cgh}{\mathcal{GH}}
\newcommand{\Homeo}{\operatorname{Homeo}}
\newcommand{\Cob}{2\operatorname{Cob}_G^{o}}
\newcommand{\Cor}{\G\operatorname{-Bord}}
\newcommand{\GOTFT}{\operatorname{GOTFT}}
\newcommand{\CSFT}{\operatorname{CSFT}}
\newcommand{\Vect}{\operatorname{Vect}}
\newcommand{\scri}{\mathscr{I}}
\newcommand{\Gen}{\G\operatorname{-Gen}}
\newcommand{\Stab}{\operatorname{Stab}}
\newcommand{\Ass}{\mathcal{ASS}}
\newcommand{\GHbord}{\GH\operatorname{-Bord}}
\newcommand{\Hcob}{H\operatorname{-Cob}_G^{o}}
\newcommand{\Gfrob}{\Delta\G\operatorname{-Frob}}

\makeatletter
\let\c@equation\c@thm
\makeatother
\numberwithin{equation}{section}

\title{Structured Topological Field Theories via Crossed Simplicial Groups}

\author{Walker H. Stern}
\address{
Max Planck Institute for Mathematics \\
Vivatsgasse 7, 53111 Bonn, Germany}
\email{walker.stern@hcm.uni-bonn.de}

\date{}
\begin{document}

\begin{abstract}

We show how the framework of crossed simplicial groups may be used to provide a classification of topological field theories on open cobordism categories defined by reductions of the structure group to a planar Lie group. Such theories are equivalent to algebras equipped with a group action and a non-degenerate trace satisfying certain invariance requirements which generalize the notion of a frobenius algebra. 

\end{abstract}

\maketitle

\tableofcontents

\newpage

\section*{Introduction}

Useful results on the classification of topological field theories have often been obtained by extracting combinatorial data from the relevant cobordism category and translating it into an algebraic framework. In particular, the study of graphs embedded in surfaces has proved to be a powerful tool in the study of such theories. For example, utilizing the connection between ribbon graphs and the moduli space of marked Riemann surfaces yields a classification of oriented two dimensional TFTs as in \cite{K}. Similar results were obtained for M\"obius graphs and unoriented two dimensional TFTs in \cite{B}, and for framed surfaces (and closed TFT's) in \cite{NR}.

The goal of this paper is to generalize these approaches using the notion of structured graphs and structured surfaces associated to a crossed simplicial group $\Delta\G$ introduced in \cite{DK}. To do so, we introduce a category of bordisms $\Cob$ whose bordisms $M$ are equipped with a reduction of the structure group 
\begin{equation*}
\rho: F_\mathbb{G}\to \operatorname{Fr}_M
\end{equation*} 
to the frame bundle of $M$ along a connective covering
\begin{equation*}
\mathbb{G}\to GL(2,\R)
\end{equation*}
which corresponds to the crossed simplicial group $\Delta\G$ as per \cite{DK}.
By further introducing the equivalent category $\Cor$, whose morphisms are structured graphs up to a notion of equivalence, we obtain a purely combinatorial framework in which to analyze the structure of TFTs defined on $\Cob$. 

The first three sections of this paper are used to review the relevant notions developed in \cite{DK}. To wit: Crossed simplicial groups, Planar Lie groups, structured surfaces, and structured graphs. We also prove several new lemmas within this framework which are necessary to the later development of the paper.

The majority of this paper is given over to analyzing the structure of the category $\Cor$, introduced in section 4. In section 5, we analyze the structure arising on a vector space $A$ that is the target of a symmetric monoidal functor 
\begin{equation*}
Z:\Cor\to \operatorname{Vect}_k
\end{equation*}
which we term a \emph{Crossed Simplicial Field Theory} (or CSFT). We find that $A$ must have the structure of a \emph{$\Delta\G$-frobenius algebra}, which generalizes the usual notion of a frobenius algebra.

Section 6 is primarily devoted to showing that, given a $\Delta\G$-frobenius algebra $A$, we can reconstruct a CSFT with target $A$, yielding the main result:
\begin{thm1}
For a balanced crossed simplicial group $\Delta\G$, there is an equivalence of categories
\begin{equation*}
\CSFT_k\cong \Gfrob_k
\end{equation*}
\end{thm1}

The rest of the paper comprises a brief digression to show that the theory developed in the previous sections can be applied with almost no alteration to equivariant TFTs, ie, to field theories on bordism categories whose morphisms come equipped with the additional datum of a principal bundle.

\section*{Acknowledgments}  I owe a great debt of thanks to my advisor, Tobias Dyckerhoff, for suggesting this topic, and for all his help, inspiration, and patience. I am also grateful to Thomas Poguntke for helpful discussions, and to the Max Planck Institute for Mathematics for supporting my studies.

\newpage

\section{Crossed Simplicial Groups}

\subsection{First Properties}
Crossed simplicial groups are objects which in some sense generalize Connes' Cyclic Category $\Lambda$. In the loosest sense, they are versions of the simplex category $\Delta$  ``decorated" with groups of automorphisms on the objects. In much the same way that one can define cyclic objects as functors $F:\Lambda^{op}\to\mathcal{C}$, there is an analogous notion for crossed simplicial group, and some crossed simplicial groups have been used to construct interesting analogs of cyclic homology (eg Dihedral or Quaternionic homology). More precisely:

\begin{defn}
A \emph{Crossed Simplicial Group} is a category $\Delta\G$ equipped with an embedding (suppressed in all following notation), $i:\Delta\to\Delta\G$ which is bijective on objects, and satisfying the property that, for every $[n],\;[m]$ there is a unique bijection
\begin{equation*}
CF: \Hom_{\Delta\G} ([n],[m]) \to \Aut_{\Delta\G}([n])\times\Hom_{\Delta} ([n],[m])
\end{equation*}
inverse to composition. This bijection is called the \emph{canonical factorization}. 

Typically, we use the notation $\G_n:=\Aut_{\Delta\G}([n])$. 
\end{defn}

We can apply the canonical factorization to pull back group elements along morphisms in $\Delta$. If $\phi\in\Hom_\Delta ([n],[m])$ and $g\in\G_m$, canonical factorization implies that there are unique morphisms $\phi^\ast g\in\G_n$ and $g^\ast \phi\in\Hom_\Delta([n],[m])$ such that TFDC
\begin{equation*}
\xymatrix{
[n]\ar[r]^{\phi^\ast g}\ar[d]_{g^\ast \phi} & [n]\ar[d]^{\phi}\\
[m]\ar[r]_{g} & [m]
}
\end{equation*}

In general, the maps of sets $\phi^\ast:\G_m\to\G_n$ will preserve identities, but are \textbf{not} group homomorphisms. However, in the special case where $[m]=[0]$, the set $\Hom_\Delta ([n],[0])$ consists of a unique morphism, $\omega_n$. Using the uniqueness of this morphism, it is immediate that the map
\begin{equation*}
\omega_n^\ast :\G_0\to\G_n
\end{equation*}
is a homomorphism. 

For any crossed simplicial group, we can trace through the canonical factorization to get the decomposition
\begin{eqnarray*}
\Hom_{\Delta\G}([0],[n]) & = & \Hom_\Delta ([0],[n])\times \G_0\\
                         & = & \{0,1,\ldots,n\}\times \G_0
\end{eqnarray*}
We can make use of this decomposition to define a functor 
\begin{equation*}
\lambda: \Delta\G \to \mathbf{FSet}
\end{equation*}
given on objects by $[n]\mapsto \Hom_{\Delta\G} ([0],[n])/\G_0$. On morphisms, we can compute the value of $\lambda$ explictly using the canonical factorization. In what follows, we denote by $CF_\Delta$ the projection of the canonical factorization onto its second component, and by $CF_\G$ the projection onto the first. We also identify $0\leq i\leq n$ with the obvious map in $\Hom_\Delta([0],[n])$. Given a morphism $\psi:[n]\to[m]$ in $\Delta\G$, we can form the commutative diagram
\begin{equation*}
\xymatrix@=4em{
[n]\ar[r]^{\psi}  & [m] \\
[0]\ar[u]^{i} \ar[r]_{CF_\G(\psi\circ i)} & [0] \ar[u]_{CF_\Delta (\psi\circ i)}
}
\end{equation*}
so that $\lambda(\psi)(i)=CF_\Delta(\psi\circ i)$. Restricting this functor  $\lambda$ to $\G_n$,  we get a group homomorphism
\begin{equation*}
\lambda_n:\G_n \to \operatorname{Isom}_{\operatorname{FSet}} (\lbrace 0,1,\ldots,n\rbrace,\lbrace 0,1,\ldots,n\rbrace)=\Sigma_{n+1}
\end{equation*} 
The kernel of this homomorphism will be denoted by $\G_n^0$.

\subsection{Examples}
Many examples of crossed simplicial groups have appeared in the literature (for a few examples, see bibliography), sometimes appearing under the name "skew-simplicial groups" (eg \cite{K}). We list a few here for convenient reference. 

\begin{exmp}
 As a first example, we can take the \emph{trivial crossed simplicial group} $\Delta$. 
\end{exmp}

\begin{exmp}
\emph{Connes' cyclic category}, which we denote $\Lambda$, following \cite{DK}. We will denote the objects, as in $\Delta$, by the symbols $[n]$. If we let $C_n$ be the standard circle with $n+1$ marked points, we can define a morphism $f:[n]\to[m]$ as a homotopy class of monotone degree 1 maps $C_n\to C_m$ mapping the set of marked points of $C_m$ into the marked points of $C_n$. The morphisms of $\Delta$ can be recovered by considering only those morphisms whose homotopy inverse sends the arc between $m$ and $0$ to the arc between $n$ and $0$. The $n$th automorphism group is $\Lambda_n=\Z/(n+1)\Z$, and the homomorphism $\lambda_n$ sends $1$ to the permuation $(0,1,\ldots,n)$   
\end{exmp}

\begin{exmp} 
The \emph{symmetric crossed simplicial group} $\Delta S$. In this case, we treat the objects $[n]$ as sets, and let the maps $\Sigma_{n+1}$ be precisely the maps on sets they represent. Composition is given by composition is given by composition in $\mathbf{FSet}$. The homomorphisms $\lambda_n$ are precisely the identities on $\Sigma_{n+1}$  
\end{exmp}

\begin{exmp}
 The \emph{braid crossed simplicial group} $\Delta B$. Objects are once again the standard ordinals, and morphisms are given by "generalized braids" as follows. 

A \emph{generalized braid} is a bipartite graph with over/under crossings as before, however, the two subsets of vertices may have different cardinality, and two edges may have the same final vertex. To have a well-defined composition law, it is necessary to also fix and record the order with which the edges enter the final vertices. An example of such a generalized braid:

\begin{center}
\begin{tikzpicture}
 \node[label=0,circle,fill=black, draw, inner sep=0pt, minimum size=4pt](p1) at (-1,0) {};
 \node[label=1,circle,fill=black, draw, inner sep=0pt, minimum size=4pt](p1) at (0,0) {};
 \node[label=2,circle,fill=black, draw, inner sep=0pt, minimum size=4pt](p1) at (1,0) {};
 \node[label=3,circle,fill=black, draw, inner sep=0pt, minimum size=4pt](p1) at (2,0) {};
 \node[label=below:0,circle,fill=black, draw, inner sep=0pt, minimum size=4pt](p1) at (0,-3) {};
 \node[label=below:1,circle,fill=black, draw, inner sep=0pt, minimum size=4pt](p1) at (1,-3) {};
\begin{knot}[
clip width=5,
clip radius=8pt ,]
\strand [line width=0.4mm] (0,0) to [out= down, in=up] (-1,-1) to [out=down,in=up] (1,-2)to [out=down,in=up] (1,-3);
\strand [line width=0.4mm] (-1,0) to [out= down, in=up] (0,-1) to [out= down, in=up] (0,-3);
\strand [line width=0.4mm] (1,0)  to [out= down, in=up] (-1,-2) to [out= down, in=135] (0,-3);
\strand [line width=0.4mm] (2,0)  to [out= down, in=up] (2,-2) to [out= down, in=45] (1,-3);
\flipcrossings{2,3}
\end{knot}
\end{tikzpicture}
\end{center}

The composition law is defined in the obvious way. For example:
\begin{center}
\begin{tikzpicture}
 \node[label=0,circle,fill=black, draw, inner sep=0pt, minimum size=4pt] at (-1,0) {};
 \node[label=1,circle,fill=black, draw, inner sep=0pt, minimum size=4pt] at (0,0) {};
 \node[label=2,circle,fill=black, draw, inner sep=0pt, minimum size=4pt] at (1,0) {};
 \node[label=below:0,circle,fill=black, draw, inner sep=0pt, minimum size=4pt] at (0,-3) {};
 \node[label=below:1,circle,fill=black, draw, inner sep=0pt, minimum size=4pt] at (1,-3) {};
\begin{knot}[
clip width=5,
clip radius=8pt ,]
\strand [line width=0.4mm] (0,0) to [out= down, in=up] (-1,-1) to [out=down,in=up] (1,-2)to [out=down,in=up] (1,-3);
\strand [line width=0.4mm] (-1,0) to [out= down, in=up] (0,-1) to [out= down, in=up] (0,-3);
\strand [line width=0.4mm] (1,0)  to [out= down, in=up] (-1,-2) to [out= down, in=135] (0,-3);
\flipcrossings{2,3}
\end{knot}
\node at (2,-1.5) {$+$};
\node[label=0,circle,fill=black, draw, inner sep=0pt, minimum size=4pt] at (3,-1) {};
 \node[label=1,circle,fill=black, draw, inner sep=0pt, minimum size=4pt] at (4,-1) {};
 \node[label=below:0,circle,fill=black, draw, inner sep=0pt, minimum size=4pt] at (3,-2) {};
 \node[label=below:1,circle,fill=black, draw, inner sep=0pt, minimum size=4pt] at (4,-2) {};
\begin{knot} 
\strand [line width=0.4mm] (3,-1) to [out= down, in=up] (4,-2);
\strand [line width=0.4mm] (4,-1) to [out= down, in=up] (3,-2);
\end{knot}
\node at  (5,-1.5) {$=$};
\node[label=0,circle,fill=black, draw, inner sep=0pt, minimum size=4pt] at (6,0) {};
 \node[label=1,circle,fill=black, draw, inner sep=0pt, minimum size=4pt] at (7,0) {};
 \node[label=2,circle,fill=black, draw, inner sep=0pt, minimum size=4pt] at (8,0) {};
 \node[label=below:0,circle,fill=black, draw, inner sep=0pt, minimum size=4pt] at (7,-4) {};
 \node[label=below:1,circle,fill=black, draw, inner sep=0pt, minimum size=4pt] at (8,-4) {};
\begin{knot}[
clip width=5,
clip radius=8pt ,]
\strand [line width=0.4mm] (7,0) to [out= down, in=up] (6,-1) to [out=down,in=up] (8,-2)to [out=down,in=up] (8,-3)to [out=down,in=up] (7,-4);
\strand [line width=0.4mm] (6,0) to [out= down, in=up] (7,-1) to [out= down, in=up] (7.1,-3) to [out=down, in=up] (8,-4);
\strand [line width=0.4mm] (8,0)  to [out= down, in=up] (6,-2) to [out= down, in=up] (6.9,-3) to [out=down, in=135] (8,-4);
\flipcrossings{2,4,5,3}
\end{knot}
\end{tikzpicture}
\end{center}

It is clear that the $n$th automorphism group is precisely the braid group $B_{n+1}$. We find a copy of $\Delta$ in $\Delta B$ by restricting to those generalized braids with no crossings. To see canonical factorization, one need only draw a horizontal line after the last crossing in the generalized braid, as in the following diagram:

\begin{center}
\begin{tikzpicture}
\draw[red] (-1.5,-2) -- (2.5,-2);
 \node[label=0,circle,fill=black, draw, inner sep=0pt, minimum size=4pt](p1) at (-1,0) {};
 \node[label=1,circle,fill=black, draw, inner sep=0pt, minimum size=4pt](p1) at (0,0) {};
 \node[label=2,circle,fill=black, draw, inner sep=0pt, minimum size=4pt](p1) at (1,0) {};
 \node[label=3,circle,fill=black, draw, inner sep=0pt, minimum size=4pt](p1) at (2,0) {};
 \node[label=below:0,circle,fill=black, draw, inner sep=0pt, minimum size=4pt](p1) at (0,-3) {};
 \node[label=below:1,circle,fill=black, draw, inner sep=0pt, minimum size=4pt](p1) at (1,-3) {};
\begin{knot}[
clip width=5,
clip radius=8pt ,]
\strand [line width=0.4mm] (0,0) to [out= down, in=up] (-1,-1) to [out=down,in=up] (1,-2)to [out=down,in=up] (1,-3);
\strand [line width=0.4mm] (-1,0) to [out= down, in=up] (0,-1) to [out= down, in=up] (0,-3);
\strand [line width=0.4mm] (1,0)  to [out= down, in=up] (-1,-2) to [out= down, in=135] (0,-3);
\strand [line width=0.4mm] (2,0)  to [out= down, in=up] (2,-2) to [out= down, in=45] (1,-3);
\flipcrossings{2,3}
\end{knot}
\node at (3,-1.5) {$=$};
\begin{scope}[shift={(5,0)}]
 \node[label=0,circle,fill=black, draw, inner sep=0pt, minimum size=4pt](p1) at (-1,0) {};
 \node[label=1,circle,fill=black, draw, inner sep=0pt, minimum size=4pt](p1) at (0,0) {};
 \node[label=2,circle,fill=black, draw, inner sep=0pt, minimum size=4pt](p1) at (1,0) {};
 \node[label=3,circle,fill=black, draw, inner sep=0pt, minimum size=4pt](p1) at (2,0) {};
 \node[label=below:0,circle,fill=black, draw, inner sep=0pt, minimum size=4pt](p1) at (0,-3) {};
 \node[label=below:1,circle,fill=black, draw, inner sep=0pt, minimum size=4pt](p1) at (1,-3) {};
 \node[,circle,fill=black, draw, inner sep=0pt, minimum size=4pt](p1) at (-1,-1.8) {};
 \node[,circle,fill=black, draw, inner sep=0pt, minimum size=4pt](p1) at (0,-1.8) {};
 \node[,circle,fill=black, draw, inner sep=0pt, minimum size=4pt](p1) at (1,-1.8) {};
 \node[,circle,fill=black, draw, inner sep=0pt, minimum size=4pt](p1) at (2,-1.8) {};
  \node[,circle,fill=black, draw, inner sep=0pt, minimum size=4pt](p1) at (-1,-2.2) {};
 \node[,circle,fill=black, draw, inner sep=0pt, minimum size=4pt](p1) at (0,-2.2) {};
 \node[,circle,fill=black, draw, inner sep=0pt, minimum size=4pt](p1) at (1,-2.2) {};
 \node[,circle,fill=black, draw, inner sep=0pt, minimum size=4pt](p1) at (2,-2.2) {};
\begin{knot}[
clip width=5,
clip radius=8pt ,]
\strand [line width=0.4mm] (0,0) to [out= down, in=up] (-1,-1) to [out=down,in=up] (1,-1.8)to [out=down,in=up] (1,-1.8);
\strand [line width=0.4mm] (-1,0) to [out= down, in=up] (0,-1) to [out=down, in=up] (0,-1.8);
\strand [line width=0.4mm] (1,0)  to [out= down, in=up] (-1,-1.8) to [out=down, in=up] (-1,-1.8);
\strand [line width=0.4mm] (2,0)  to [out= down, in=up] (2,-1.8) to [out=down, in=up] (2,-1.8);
\strand [line width=0.4mm] (-1,-2.2) to [out=down, in=135] (0,-3);
\strand [line width=0.4mm] (0,-2.2) to [out=down, in=up] (0,-3);
\strand [line width=0.4mm] (1,-2.2) to [out=down, in=up] (1,-3);
\strand [line width=0.4mm] (2,-2.2) to [out=down, in=45] (1,-3);
\flipcrossings{2}
\end{knot}
\end{scope}
\end{tikzpicture}
\end{center}

The homomorphism $\lambda_n$ can be realized by forgetting crossings and viewing a braid as a permution of the $(n+1)$-element set.

For more information, see for example \cite{Kr} or \cite{FL}
\end{exmp}

\subsection{Canonical Parity}

\begin{defn}
We call a crossed simplicial group \emph{semi-constant} if the maps $\omega_n^\ast$ are all isomorphisms.
\end{defn}

\begin{rem}
Any crossed simplicial group contains the semiconstant crossed simplicial group associated to $\G_0$.
\end{rem}

\begin{defn}
For a group $G$ and a small category $\mathcal{C}$, we define an \emph{action} of $G$ on $\mathcal{C}$ to be a homomorphism
\begin{equation*}
G\to \operatorname{Aut}_{\operatorname{Cat}}(\mathcal{C})
\end{equation*}

For $G$ acting on $\mathcal{C}$ we can define the \emph{semidirect product}  $G\ltimes \mathcal{C}$ to be the category with objects the objects of $\mathcal{C}$ and morphisms $(g,\phi):g.x\to y$ for $\phi\in\Hom_\mathcal{C} (x,y)$ and $g\in G$ which compose according to the law
\begin{equation*}
(g,\phi)\circ (h,\psi)=(gh,\phi\circ g \circ \psi)
\end{equation*}
\end{defn}

\begin{prop}
Any semi-constant crossed simplicial group $\Delta\G$ is isomorphic to the semidirect product $\G_0\ltimes \Delta$.
\end{prop}
\begin{proof}
(\cite{DK} proposition 1.12)
\end{proof}

Since $\text{Aut}_\text{Cat} (\Delta)$ is generated by an involution $\daleth$, we have that for any crossed simplicial group, the action of $\G_0$ on $\Delta$ endows $\G_0$ with a \emph{Canonical Parity}, that is, a homomorphism
\begin{equation*}
\G_0\to \Z/2\Z
\end{equation*}

\begin{defn}
For a crossed simplicial group $\Delta\G$  and an algebra $A$, we define a \emph{twisted action} of $\G_0$ on $A$ to be an action whereby the even elements of $\G_0$ act by automorphisms, and the odd elements of $\G_0$ act by anti-automorphisms.  
\end{defn}

\section{Structured Sets and Structured Graphs}

A key use of the formalism surrounding crossed simplicial groups is its ability to provide interesting algebraic and combinatorial structures for study. For our purposes, the most useful of these are the (related) notions of structured sets and structured graphs, which lead to a connection between the theory of crossed simplicial groups and the theory of operads.

\begin{defn}
A \emph{$\Delta\G$-structure} on a set $I$ of cardinality $n+1$ consists of the following data:
\begin{enumerate}
\item A right $\G_n$-torsor $\mathcal{O}(I)$.
\item A map 
\begin{equation*}
\rho: \mathcal{O}(I)\longrightarrow \text{Isom}_\text{FSet} (\lbrace0,1,\ldots,n\rbrace, I)
\end{equation*}
equivariant along the homomorphism $\lambda_n: \G_n\to \Sigma_{n+1}$.
\end{enumerate}
The elements of the torsor $\mathcal{O}(I)$ are called \emph{(structured) frames} of $I$.
\end{defn}

We can, in point of fact, form a category $\cG$ whose objects are precisely $\Delta\G$-structured sets, and whose morphisms are morphisms of structured sets as follows:

\begin{defn}
A \emph{morphism of $\Delta\G$-structured sets} 
\begin{equation*}
\psi: (\mathcal{O}(I^\prime),I^\prime,\rho^\prime)\to(\mathcal{O}(I),I,\rho)
\end{equation*}
where $\vert I \vert=n+1$ and $\vert I^\prime \vert=n^\prime+1$ is given by a collection
\begin{equation*}
\left\lbrace \psi_{f,f^\prime}\in\Hom_{\Delta\G}([n^\prime],[n]) \mid f^\prime\in\mathcal{O}(I^\prime)\;\; f\in \mathcal{O}(I) \right\rbrace
\end{equation*}
such that, for any $g^\prime\in\G_{n^\prime}$ and $g\in\G_{n}$ we have that
\begin{equation*}
\psi_{f.g,f^\prime .g^\prime}=g^{-1}\circ \psi_{f,f^\prime}\circ g^\prime
\end{equation*}
Composition of morphisms is defined via the formula
\begin{equation*}
(\psi\circ \phi)_{f,f^{\prime\prime}}=\psi_{f,f^\prime}\circ\phi_{f^\prime,f^{\prime\prime}}
\end{equation*}
\end{defn}

\begin{rem}
The composition is well defined precisely because of the equivariance condition, ie, for $f^\prime, h^\prime\in \mathcal{O}(I^\prime)$, we can choose $g^\prime\in\G_{n^\prime}$ such that $h^\prime=f^\prime.g^\prime$. Then we have that
\begin{eqnarray*}
(\psi\circ \phi)_{f,f^{\prime\prime}} & = & \psi_{f,h^\prime}\circ\phi_{h^\prime,f^{\prime\prime}}\\
& = & \psi_{f,f^\prime.g^\prime}\circ\phi_{f^\prime.g^\prime,f^{\prime\prime}}\\
& = & \psi_{f,f^\prime}\circ g^\prime \circ\left(g^\prime\right)^{-1}\circ \phi_{f^\prime,f^{\prime\prime}}\\
& = & \psi_{f,f^\prime}\circ\phi_{f^\prime,f^{\prime\prime}}
\end{eqnarray*}
\end{rem}

\begin{rem}
We can define a functor $\epsilon:\Delta\G\to\cG$ defined on objects by sending the object $[n]$ to the \emph{standard $\Delta\G$-structured set} $\epsilon[n]=(\G_n,\lbrace0,1,\ldots,n\rbrace,\lambda_n)$. 

Since the datum determinining a morphism $\lbrace \psi_{f,f^\prime}\rbrace$ is determined by the equivariance condition and any one of it's members, we may associate the morphism $\lbrace \psi_{f,f^\prime}\rbrace:\epsilon[n]\to\epsilon[m]$ with the component of both the identities 
\begin{equation*}
\psi_{id_[n],id_[m]}\in\Hom_{\Delta\G}([n].[m])
\end{equation*}. 
The value of $\epsilon$ on morphisms is then given by this correspondence, making the functor fully faithful. 

Moreover, given a structured set $(\mathcal{O}(I),I,\rho)$, we can trivialize the torsor by choosing an arbitrary structured frame $f$. This yields an association of $\mathcal{O}(I)$ with $\G_n$, and, through the isomorphism $\sigma=\rho(f)$, of $[n]$ with $I$. Under these associations, we see that $\rho$ becomes precisely the map $\lambda_n$, so that $\epsilon$ is, in fact, an equivalence of categories.
\end{rem}

\begin{exmp}
The automorphism groups of the trivial crossed simplicial group $\Delta$ are themselves trivial, so that choosing a torsor $\mathcal{O}(I)$ and a map $\rho$ is equivalent to choosing a linear order on the set $I$. Morphisms are precisely those morphisms which preserve this linear order, so  the category of $\Delta$-structured sets is precisely the category $\boldsymbol{\Delta}$, which is a well-known equivalent to $\Delta$.
\end{exmp}

\begin{exmp}
If we explore the category of structured sets corresponding to the cyclic category $\Lambda$, we see that a choice of $\Lambda$-structure on a set $I$ corresponds to a choice of $n+1$ linear orders on $I$ which are related to one another by cyclic permutations of the labels; in other words: a cyclic order on $I$. Morphisms are maps of sets which preserve the cyclic order. 
\end{exmp}

For any morphism $\psi$ in $\cG$, we can also talk about the induced map on sets. Much like $\Delta\G$, the category $\cG$ admits a functor to $\mathbf{FSet}$, as described below (compare \cite{DK}, proposition 2.3).

\begin{prop}
Given a $\Delta\G$-structured set $(\mathcal{O}(I),I,\rho)$, there are canonical identifications
\begin{eqnarray*}
\mathcal{O}(I) & \cong &  \operatorname{Isom}_\cG (\epsilon[n],(\mathcal{O}(I),I))\\
I & \cong & \Hom_\cG (\epsilon[0],(\mathcal{O}(I),I)
\end{eqnarray*}
so that, in particular, there is a functor $\lambda_\cG:\cG\to \mathbf{FSet}$ extending $\lambda:\Delta\G\to\mathbf{FSet}$.
\end{prop}

\begin{rem}
We can explicitly compute $\lambda_\cG$ on morphisms using the fact that it is an extension of $\lambda$, ie that TFDC,
\begin{equation*}
\xymatrix{
\Delta\G\ar[rr]^\epsilon\ar[dr]_{\lambda} & & \cG\ar[dl]^{\lambda_\cG} \\
 & \mathbf{FSet} &
}
\end{equation*}
We then see that it takes the value (independent of all choices made) of
\begin{equation*}
\lambda_\cG (\lbrace \psi_{f,f^\prime}\rbrace)=\rho^\prime(f^\prime)^{-1}\circ \lambda(\psi_{f,f^\prime}) \circ \rho(f)
\end{equation*}
\end{rem}

\subsection{Structured Graphs}

Using the formalism of $\Delta\G$-structured sets, we can further define $\G$-structured graphs. To begin, we recall the following definition of a graph.

\begin{defn}
A \emph{graph} $\Gamma$ is given by 
\begin{enumerate}
\item A set $V$ of vertices
\item A set $H$ of half-edges
\item An involution $\mathfrak{y}$ on $H$
\item A map $s:H\to V$
\end{enumerate}

The set $H(v):=s^{-1}(v)$ for $v\in V$ is the set of half-edges incident to $v$. The \emph{edges} of $\Gamma$ are the sets $\lbrace h,\mathfrak{y}(h)\rbrace$ for $h\in H$. If $h$ is a fixed point of $\mathfrak{y}$, we call it an \emph{external half-edge} of $\Gamma$. 
\end{defn}

\begin{defn}
The \emph{incidence category} $I(\Gamma)$ corresponding to a graph $\Gamma$ is the category with objects the edges and vertices of $\Gamma$, and, for every edge $e$ incident to a vertex $v$, a morphism $v\to e$.  The \emph{incidence diagram} $I_\Gamma:I(\Gamma)\to \mathbf{Set}$ is the functor that assigns to every vertex $v$ the set $H(v)$, and to every edge $e=\lbrace h,\mathfrak{y}(h)\rbrace$ the set $\lbrace h,\mathfrak{y}(h)\rbrace$. For a half-edge $h$ incident to $v$, $I_\Gamma$ sends the morphism $v\to e$ to the map $H(v)\to \lbrace h,\mathfrak{y}(h)\rbrace$ which sends $h$ to $h$, and collapses all other half-edges to $\mathfrak{y}(h)$.
\end{defn}

\begin{rem}
We can recover the conventional image of a graph by taking the geometric realization of the nerve of the incidence category. This simplicial complex will have a zero simplex for every edge and vertex, and a one simplex for every half edge. We use the notation $\vert\Gamma\vert:=\vert N(I(\Gamma))\vert$.
\end{rem}

\begin{defn}
For a crossed simplicial group $\Delta\G$ with corresponding category of structured sets $\cG$, a $\Delta\G$-structure on a graph $\Gamma$ is a lift $\tilde{I}_\Gamma$ of the incidence diagram $I_\Gamma$ to $\cG$ as in the diagram 
\begin{equation*}
\xymatrix{
 & \cG\ar[d]^{\lambda_\cG} \\
 I(\Gamma)\ar[r]_{I_\Gamma}\ar[ur]^{\tilde{I}_\Gamma} & \mathbf{FSet}
}
\end{equation*}
\end{defn}

\begin{defn}
An \emph{augmented $\Delta\G$-structured graph} is a $\Delta\G$-structured graph with the additional datum of a morphism of structured sets $\phi_h:H(v)\to \epsilon[1]$ for every external half-edge $h$ incident to $v$, satisfying the condition that $\phi_h^{-1}(\phi_h(h))=\lbrace h\rbrace$. 

We call an external half-edge $h$ with augmentation \emph{incoming} if $\phi_h(h)=1$ and \emph{outgoing} if $\phi_h(h)=0$
\end{defn}

\begin{rem}
We can concatenate two augmented $\Delta\G$-structured graphs $\Gamma$ and $\Gamma^\prime$ at an outgoing half-edge $h$ of $\Gamma$ and an incoming half-edge $h^\prime$ of $\Gamma^\prime$ to yield a new structured graph. We do this by identifying 1 (in the underlying set of $\epsilon[1]$)  with $h^\prime$ and $0$ with $h$, and taking the new graph to be (in a slight abuse of notation)
\begin{equation*}
\Gamma\overset{\phi_h}{\rightarrow} \epsilon[1] \overset{\phi_{h^\prime}}{\leftarrow} \Gamma^\prime
\end{equation*}
\end{rem}

\begin{defn}
We say that two graphs $\Gamma$ and $\Gamma^\prime$ are \emph{equivalent} if there exists a functor $\phi:I(\Gamma)\to I(\Gamma^\prime)$ such that the induced map $\vert\phi\vert:\vert\Gamma\vert\to\vert\Gamma^\prime\vert$ is a homotopy equivalence.

We say that two $\Delta\G$-structured graphs $\Gamma$ and $\Gamma^\prime$ are equivalent if they are equivalent as graphs via an equivalence $\phi$, and if the pullback morphism $\phi_\ast I_\Gamma^\prime\to I_\Gamma$ admits a lift to a morphism of $\cG$-diagrams  $\tilde{\phi}_\ast \tilde{I}_\Gamma^\prime\to \tilde{I}_\Gamma$

We will sometimes refer to an equivalence of (structured) graphs as a \emph{contraction}
\end{defn}

For any $\Delta\G$-structured graph, there is, in fact, a canonical contraction that contracts a single edge
\begin{equation*}
H(v)\rightarrow \lbrace h,h^\prime\rbrace\leftarrow H(v^\prime)
\end{equation*}
to a single vertex. The functor of incidence categories is obvious, and the pullback lift of the pullback functor is simply given by the cone diagram of the limit of the diagram corresponding to the edge in question. To make sure this is well defined, we need the following.

\begin{lem}
Any diagram in a crossed simplicial group $\Delta\G$ of the form
\begin{equation*}
\xymatrix{
 & [m]\ar[d]^i\\
 [n]\ar[r]_j & [1]
}
\end{equation*}
such that $i$ maps all but one element to $0$ and $j$ maps all but one element to $1$, can be completed to a pullback diagram
\begin{equation*}
\xymatrix{
 [m+n-1]\ar[r]\ar[d] & [m]\ar[d]^i\\
 [n]\ar[r]_j & [1]
}
\end{equation*}
\end{lem}

\begin{proof}
We first observe that this can be reduced to a statement about diagrams in $\Delta$. To see this, for the bottom maps to satisfy the requisite conditions, the must be given by
\begin{eqnarray*}
i & = & \psi_0\circ g \\
j & = & \psi_m\circ h
\end{eqnarray*}
respectively. Here $\psi_0$ is the map in $\Delta$ sending everything except $0$ to $1$ and $\psi_m$ is the map in $\Delta$ sending everything except $m$ to $0$. 
This means that we can take the existence condition
\begin{equation*}
\xymatrix{
 [k]\ar@/^/[drr]^{z} \ar@/_/[ddr]_{\ell}\ar@{-->}[dr]^{\exists !} & & \\
  & [m+n-1]\ar[r]_p \ar[d]_q & [m]\ar[d]^i\\
  & [n]\ar[r]_j & [1]
}
\end{equation*}
and replace it with
\begin{equation*}
\xymatrix{
 [k]\ar@/^/[drr]^{h\circ z} \ar@/_/[ddr]_{g\circ \ell}\ar@{-->}[dr]^{\exists !} & & \\
  & [m+n-1]\ar[r]_{h\circ p} \ar[d]_{g\circ q} & [m]\ar[d]^{\psi_m}\\
  & [n]\ar[r]_{\psi_0} & [1]
}
\end{equation*}

However, by Canonical Factorization, we see that we can rewrite $g\circ q=\theta\circ g^\prime$ and $h\circ p=h^\prime \circ \gamma$. The central square commuting means that $g^\prime= h^\prime$ and $\psi_0\circ \theta=\psi_m\circ \gamma$. Hence, we see that the condition above is equivalent to the condition that 
\begin{equation*}
\xymatrix{
 [k]\ar@/^/[drr]^{h\circ z} \ar@/_/[ddr]_{g\circ \ell}\ar@{-->}[dr]^{\exists !} & & \\
  & [m+n-1]\ar[r]_{\gamma} \ar[d]_{\theta} & [m]\ar[d]^{\psi_m}\\
  & [n]\ar[r]_{\psi_0} & [1]
}
\end{equation*}

Similarly, taking the canonical factorizations $h\circ z=\beta\circ t$ and $g\circ \ell=\alpha \circ s$, we see that commutativity means that $s=t$ and the $\Delta$-morphisms commute. Hence, replacing $[k]$ through the isomorphism $t$, we see that this is equivalent to the existence condition
\begin{equation*}
\xymatrix{
 [k]\ar@/^/[drr]^{\beta} \ar@/_/[ddr]_{\alpha}\ar@{-->}[dr]^{\exists !} & & \\
  & [m+n-1]\ar[r]_{\gamma} \ar[d]_{\theta} & [m]\ar[d]^{\psi_m}\\
  & [n]\ar[r]_{\psi_0} & [1]
}
\end{equation*}
However, this means that the map whose existence we are asserting must also be a morphism in $\Delta$. ie, diagrams of the form above admit such a pullback in $\Delta\G$ if and only if diagrams of the form 
\begin{equation*}
[n]\overset{\psi_0}{\rightarrow} [1] \overset{\psi_m}{\leftarrow} [m]
\end{equation*}
admit a pullback in $\Delta$. However, we know that there is a crossed simplicial group, $\Lambda$, in which such a pullback exists (see \cite{D} Lemma 1.9). Therefore, it exists in any crossed simplicial group.
\end{proof}

\begin{rem}
A similar argument holds for the case where one or both of $[m]$ and $[n]$ are $[0]$. The pullback in $\Delta$ can be computed explicitly, and is again equivalent to the pullback in $\mathbf{FSet}$. 
\end{rem}

\begin{cor}
Any diagram of $\Delta\G$-structured sets
\begin{equation*}
H(v)\rightarrow \lbrace h,h^\prime\rbrace\leftarrow H(v^\prime)
\end{equation*}
admits a limit in $\cG$ whose underlying set is given by $\left( H(v)\smallsetminus\lbrace e\rbrace\right)\cup\left( H(v^\prime)\smallsetminus\lbrace e\rbrace\right)$
\end{cor}

\subsection{$\Delta\G$-structured trees and operads}

Given an augmented $\Delta\G$-structured graph $\Gamma$, our notion of contraction always allows us to find an equivalent $\Delta\G$-structured graph with a single vertex. This and the concatenation operation sketched above, suggest that it may be possible to extract an operad defined by $\Delta\G$ from the notion of $\Delta\G$-structured graphs. 

\begin{defn}
We call an augmented $\Delta\G$-structured graph $\Gamma$ a ($\Delta\G$-structured augmented) tree if $\vert\Gamma\vert$ is simply connected. Let $T_{\Delta\G}(n)$ denote the category of $\Delta\G$-structured augmented trees with external half-edges labelled by the numbers $\lbrace 0,1,\ldots,n\rbrace$, such that the edge $0$ is outgoing, and all other edges are incoming. The morphisms are given by contractions. Letting $\mathcal{T}_{\Delta\G}(n)$ be the groupoid completion of this collection, we define 
\begin{equation*}
P_{\Delta\G}(n)=\pi_0 \mathcal{T}_{\Delta\G}(n)
\end{equation*}
\end{defn}

We then have the following results (\cite{DK} propositions 4.16 and 4.18)

\begin{prop}\label{prop:monoid}
The monoid $P_{\Delta\G}(1)$ can be canonically identified with the group $\G_1^0$.
\end{prop}

\begin{proof}
We can represent an equivalence class in $P_{\Delta\G}(1)$ by an augmented graph of the form
\begin{equation*}
[1]\overset{\phi_e}{\leftarrow} H(v) \overset{\phi_f}{\rightarrow} [1]
\end{equation*}
with morphisms $\phi_e$ (incoming) and $\phi_f$ (outgoing). We can then identify $H(v)$ with $[1]$ via the isomorphism $\phi_e$, giving an equivalent graph of the form 
\begin{equation*}
[1]\overset{id}{\leftarrow} [1] \overset{g}{\rightarrow} [1]
\end{equation*}
where the set map underlying $g$ is the identity. This means precisely that $g\in\G_1^0$. 

To see compatibility with the group law, we notice that composing two such intervals, we get
\begin{equation*}
[1]\overset{id}{\leftarrow} [1] \overset{g}{\rightarrow} [1]\overset{id}{\leftarrow} [1] \overset{h}{\rightarrow} [1]
\end{equation*}
which admits a contraction to 
\begin{equation*}
[1]\overset{id}{\leftarrow} [1] \overset{hg}{\rightarrow} [1]
\end{equation*}
So that composition and contraction gives the same result as multiplication.
\end{proof}

\begin{prop}
For any crossed simplicial group $\Delta\G$, the action of $\Sigma_n$ on the incoming half-edges by relabeling, along with the maps
\begin{equation*}
P_{\Delta\G}(n)\times P_{\Delta\G}(a_1)\cdots\times P_{\Delta\G}(a_n)\to P_{\Delta\G}(\sum a_i)
\end{equation*}
gives $P_{\Delta\G}:=\left\lbrace P_{\Delta\G}(n)\right\rbrace_{n\geq1}$ the structure of an operad in the category of sets.
\end{prop}

\begin{exmp}
Consider the operad $P_{\Delta}$ associated to the trivial crossed simplicial group. The only morphisms in $\Hom_{\Delta}([n],[1])$ which could be used to define augmentations are 
\begin{enumerate}
\item The morphism $\psi$ given by collapsing everything except $n$ to $0$.
\item The morphism $\phi$ given by collapsing everything except $0$ to $1$.
\end{enumerate}
As a result, the only non-zero component of the operad $P_\Delta$ is $P_\Delta(1)$, and it contains precisely one member. 

The algebras over this operad (in a category $\mathcal{C}$) are precisely equivalent to objects in $\mathcal{C}$. The operad itself is the initial object in the category of operads in $\mathbf{FSet}$.
\end{exmp}

\begin{exmp}
If we take $P_\Lambda$, the operad associated with the cyclic category, then we have a more interesting structure. As we will see in the next section, $P_\Lambda$ is precisely the associative operad.
\end{exmp}

\section{Planar and Balanced Crossed Simplicial Groups}
\subsection{Planar Crossed Simplicial Groups and Topological Motivation}

A useful class of crossed simplicial groups is provided by the study of surfaces equipped with a reduction of the structure group. 
\begin{defn}
A \emph{connective covering} of a Lie group $G$ is a morphisms of Lie groups 
\begin{equation*}
p:\tilde{G}\to G
\end{equation*}
such that $p$ is a covering of its image, and the pre-image of the compnent of the identity is connected.
\end{defn}

\begin{defn}
A \emph{Planar Lie Group} is a connective covering of $O(2)$.  
\end{defn}

As it turns out, there is a direct correspondence between a certain class of crossed simplicial groups and planar Lie groups (listed in figure 1).

\begin{defn}
A crossed simplicial group $\Delta G$ is called \emph{planar} if it corresponds to a planar Lie group $G$.
\end{defn}

\begin{figure}[h!]
\begin{tabular}{|c | c | c |}
\hline
CSG & 
Groups $\G_n$ &
Planar Lie Group $G$ \\ \hline \hline
Cyclic category $\Lambda$ &
$\Z/(n+1)\Z$ &
$SO(2)$ \\ \hline
Dihedral Category $\Xi$ &
$D_{n+1}$ &
$O(2)$\\ \hline
Paracyclic Category $\Lambda_\infty$ & $\Z$ & $\tilde{SO}(2)$ \\ \hline
Paradihedral Category $\Xi_\infty$ & $D_\infty$ & $\tilde{O}(2)$\\ \hline
N-cyclic Category $\Lambda_N$ & $\Z/N(n+1)\Z$ & $\operatorname{Spin}(2)_N$ \\ \hline
N-dihedral Category $\Xi_N$ & $D_{N(n+1)}$ & $\operatorname{Pin}_N^+ (2)$ \\ \hline
M-Quaternionic Category $\nabla_M$ & $Q_{M(n+1)}$ & $\operatorname{Pin}_{2M}^- (2)$\\ \hline
\end{tabular}
\caption{Correspondence between crossed simplicial groups and planar Lie groups. Further exposition can be found in \cite{DK} section 1.4.}
\end{figure}

We can better explain this correspondence by citing a result of Fiedorowicz and Loday \cite{FL}.

\begin{prop}
The geometric realization of the simplicial set $\G_\bullet$ underlying a crossed simplicial group $\Delta\G$ has the structure of a topological group $G=\vert \G \vert$. Moreover, if $\Delta\G$ is planar, then $G$ is the planar Lie group corresponding to $\Delta\G$ in figure 1.
\end{prop}

As it happens, planar crossed simplicial groups are also reflexive. We can see this by appealing to an alternate topological model for such groups. 

\begin{prop}
Let $\operatorname{Con}(G)$ be the category of connective coverings of a topological group $G$. Then a homotopy equivalence $G\sim K$ gives rise to an equivalence of categories $\operatorname{Con}(G)\cong\operatorname{Con}(K)$. 
\end{prop}
\begin{proof}
(\cite{DK}, proposition 1.30)
\end{proof}

\begin{cor}
A planar Lie group $G$ uniquely corresponds to a connective covering 
\begin{equation*}
p_{Homeo}: \Homeo^G(S^1)\to \Homeo(S^1)
\end{equation*}
and a connective covering
\begin{equation*}
p_\mathbb{G}: \mathbb{G}\to GL(2,\R)
\end{equation*}
via the homotopy equivalences 
\begin{equation*}
O(2)\to GL(\R,2)\to \Homeo(S^1)
\end{equation*}
\end{cor}

Using this corollary, we can find an alternative model for $\Delta\G$-structured sets.

\begin{defn}
We define a \emph{Marked Circle} to be a pair $(C,J)$ where $C$ is a topological space homeomorphic to $S^1$ and $J$ is a closed subset of $C$ which is homeomorphic to a disjoint union of a finite number of intervals. 

A \emph{morphism} of marked circles from $(C,J)$ to $(C^\prime,J^\prime)$ is an element of the set
\begin{equation*}
\Homeo ((C,J),(C^\prime,J^\prime))=\left\lbrace \phi\in\Homeo(C,C^\prime) \big\vert \phi(J)\subset(J^\prime) \right\rbrace
\end{equation*}

A \emph{$G$-Structured Marked Circle} is a circle $C$ equipped with a marking $J$ 
and a reduction of the structure group $\rho:F\to \Homeo(S^1,C)$ 
equivariant along $p_{\Homeo}$. 
Further, we define $\Homeo^G ((C,J,\rho),(C^\prime,J^\prime,\rho^\prime))$ to be the set of pairs $(\phi,\tilde{\phi})$ where $\phi\in\Homeo((C,J),(C^\prime,J^\prime))$ and $\tilde{\phi}$ is a lift of $\phi$ such that TFDC:
\begin{equation*}
\xymatrix{
F\ar[r]^{\tilde{\phi}}\ar[d]_\rho & F^\prime\ar[d]^{\rho^\prime}\\
\Homeo(S^1,C)\ar[r]_\phi & \Homeo(S^1,C^\prime)
}
\end{equation*}
\end{defn}

\begin{defn}
We define the \emph{category of structured circles}  $\mathcal{C}_G$ to be the category with objects given by marked structured circles and hom-sets given by
\begin{equation*}
\Hom_{\mathcal{C}_G}\left( (C,J,\rho),(C^\prime,J^\prime,\rho^\prime) \right)= \pi_0 \Homeo^G \left((C,J,\rho),(C^\prime,J^\prime,\rho^\prime)\right)
\end{equation*}
\end{defn}

This category admits a functor $\lambda_{\mathcal{C}_G}:\mathcal{C}_G\to\mathbf{FSet}$ given by sending $(C,J,\rho)$ to $\pi_0(J)$. On morphisms, it is given by the induced map on connected components of $\phi$. 

\begin{prop}
Let $G$ be a planar Lie group corresponding to the crossed simplicial $G$. Then there is an equivalence of categories $\pi:\mathcal{C}_G\to \mathcal{G}$ such that the functor $\lambda_{\mathcal{C}_G}$ factors as
\begin{equation*}
\xymatrix{
\mathcal{C}_G\ar[r]^{\pi}\ar[ddr]_{\lambda_{\mathcal{C}_G}} & \mathcal{G}\ar[dd]_{\lambda_{\mathcal{G}}} & \Delta\G\ar[l]^\cong \ar[ddl]_{\lambda}\\
  & & \\
  & \mathbf{FSet} & 
}
\end{equation*}
\end{prop}

\begin{proof}
(\cite{DK}, Theorem 2.13)
\end{proof}

\begin{rem}
Though we will not reproduce the proof here, it will be of use to briefly write down the set and torsor associated to a structured marked circle $(C,J)$, $F\to\Homeo(S^1,C)$.  The set in question will simply be
\begin{equation*}
I=\pi_0 (J)
\end{equation*}

To define the `torsor,' Let $\Homeo((S^1,[n]),(C,J))$ be the subspace of $\Homeo(S^1,C)$ which maps the standard set of $n+1$ marked points (roots of unity) bijectively to the intervals comprising  $J$. Let $F_J$ be the restriction of the bundle $F$ to $\Homeo((S^1,[n]),(C,J))$. Then $\pi_0(F_J)$ can be given a canonical $\G_n$-torsor structure such that the obvious map 
\begin{equation*}
\pi_0(F_J)\to\text{Isom}_{\mathbf{FSet}}([n],I)=\pi_0\left( \Homeo\left((S^1,[n]),(C,J)\right)\right)
\end{equation*} 
is equivariant. 
\end{rem}

\begin{lem}\label{lem:dualitylem}
There is a duality $D$ on $\mathcal{C}_G$, which sends the equivalence class of $[n]$ to the equivalence class of $[n]$. 
\end{lem}

\begin{proof}
We can write down this duality explicitly. On objects, we define 
\begin{equation*}
(C,J,\rho)\mapsto (C,C\setminus J,\rho)
\end{equation*}
And on morphisms, 
\begin{equation*}
(\phi,\tilde{\phi})\mapsto (\phi^{-1},\tilde{\phi}^{-1})
\end{equation*}
One can quickly verify that this is functorial, and it is obvious that $D^2=Id$. Notably, the set $C\setminus J$ is comprised of the same number of intervals as the set $J$ itself. 
\end{proof}

The final useful connection realized by this topological framework is that of graphs. Given a surface with a specific sort of structure, we can obtain a $\Delta\G$-structured graph. 

\begin{defn}
Let $p:G\to O(2)$ be a planar Lie group corresponding to a connective covering $\mathbf{p}:\mathbb{G}\to GL(\R,2)$.
A \emph{$G$-structured surface} is a surface $S$ equipped with a reduction of the structure group
\begin{equation*}
 \rho:F\to \text{Fr}_S
\end{equation*}
  along the map $\mathbf{p}:\mathbb{G}\to GL(\R,2)$, where $\text{Fr}_S$ denotes the frame bundle of the tangent bundle of $S$. 

  A structured diffeomorphism of structured surfaces is simply a diffeomorphism together with a lift to the principal
  $\mathbb{G}$-bundles. 
  \end{defn}
  
 We then have the following result (proposition 4.8, \cite{DK}):
 
 \begin{prop}\label{prop:structgraphprop}
  Let $S$ be a $G$-structured surface and $\Gamma$ be a graph embedded in $S$. Then $\Gamma$ is endowed with a canonical 
  $\Delta\G$-structure by the embedding.
 \end{prop}

\subsection{Balanced Crossed Simplicial Groups}

Several properties displayed by planar crossed simplicial groups have a fundamental bearing on the construction of topological field theories. So much so, in fact, that these properties on their own suffice to allow to construct a combinatorial version of a topological field theory (a `Crossed Simplicial Field Theory,' as defined in section 4).

For any crossed simplicial group $\Delta\G$, the morphisms in $\Delta$ 
\begin{equation*}
i_n: [0]\to[n]
\end{equation*}
given by sending $0$ to $i$ define pullback maps:
\begin{equation*}
i_n^\ast:\G_n\to \G_0
\end{equation*}
under canonical factorization. In particular, if we restrict the source to $\Stab(i)\subset\G_n$, we can immediately verify that these maps become group homomorphisms. 

\begin{defn}
A crossed simplicial group $\Delta\G$ is called \emph{balanced} if
\begin{itemize}
\item $\Delta\G$ admits a duality 
\begin{equation*}
D_\G: \Delta\G\overset{\cong}{\longrightarrow} \Delta\G^{\text{op}}
\end{equation*}
such that, denoting by $\{i,j\}$ the map in $\Hom_\Delta([1],[n])$ sending $0$ to $i$ and $1$ to $j$, we have 
\begin{equation*}
D_\G(\{i-1,i\})=\psi_i \;\;\;\; 0<i\leq n
\end{equation*}
where $\psi_i^{-1}(i)=\{i\}$ and 
\begin{equation*}
D_\G(\{0,n\})=\phi
\end{equation*}
where $\phi^{-1}(0)=\{0\}$.
\item The pullback maps 
\begin{equation*}
i_n^\ast:\Stab(i)\to \G_0
\end{equation*}
are all isomorphisms. 
\item $1_1^\ast=0_1^\ast$ on $\Stab(1)=\Stab(0)$.
\end{itemize}
\end{defn}

\begin{lem}
Any planar crossed simplicial group $\Delta\G$ is balanced.
\end{lem}

\begin{proof}
The first property is immediate from lemma \ref{lem:dualitylem}, with one slight caveat: We need to choose a specific trivialization of the structured set extracted from the dual circle, such that the duality has the desired property on the morphisms mentioned in the definition. 

Suppose we have a structured circle $(C,I,\rho)$, with a trivialization $x\in\pi_0\left(F_I\right)$ with 
\begin{equation*}
\rho_{\pi_0(I)}(x)=A
\end{equation*}
For an interval $A\subset I$. Then, letting $I^\vee$ be the interstices of the marked circle, consider a point $z$ in the interval $B\subset I^\vee$ directly anti-clockwise from $A$. Let $J$ be an interval containing $A$ and $z$. Then $x$ defines a trivialization of $\pi_0(F_J)$, which in turn defines a trivialization of $\pi_0\left(F(I^\vee)\right)$. It is easy to verify that this procedure yields the desired properties.

To verify the second property, we again resort to the structured circle model. Let $(C,I,\rho)$ be a structured circle with one marked interval, and let $(\tilde{C},J,\tilde{\rho})$ be a structured circle with $n$ marked intervals. Choose a morphism of structured circles
\begin{equation*}
\Phi=(\phi,\tilde{\phi}):(C,I,\rho)\to(\tilde{C},J,\tilde{\rho})
\end{equation*}
which sends $I$ into an interval $A\subset J$. We can identify this morphism with $i_n$ in $\Delta$ by choosing connected components $x\in \pi_0\left( F_I\right)$ and $y\in \pi_0\left(F_A \right)$.

An automorphism $(\gamma,\tilde{\gamma})$ of $(\tilde{C},J,\tilde{\rho})$ which fixes $A\subset J$ can be uniquely specified by the connected component of $F_A$ to which $\tilde{\gamma}$ sends $y$. Similarly, we can specify an automorphism $(\delta,\tilde{\delta})$ of $(C,I,\rho)$ by specifying the connected component of $F_I)$ to which $\tilde{\delta}$ sends $x$. Since $\tilde{\phi}$ is a bundle map, it induces an bijection 
\begin{equation*}
\pi_0\left( F_I\right)\to \pi_0\left(F_A \right)
\end{equation*}
This bijection is precisely the map $i_n^\ast$, proving the second part of the lemma. 

The equality $1_1^\ast=0_1^\ast$ can be checked case-by-case. 
\end{proof}

We call the special case $\Stab(1)=\Stab(0)\subset \G_1$, $\G_1^0$. For a balanced crossed simplicial group, we have a canonical identification
\begin{equation*}
\G_0^1\cong\G_0
\end{equation*}
As a result, we have the following:

\begin{lem}
For a balanced crossed simplicial group, there is a canonical identification $P_{\Delta\G}(1)\cong\G_0$
\end{lem}

We also immediately obtain another proposition relevant to structured graphs:

\begin{lem}\label{lem:augpullback}
For a balanced crossed simplicial group $\Delta\G$, an augmentation map 
\begin{equation*}
\phi: [n]\to [1]
\end{equation*}
gives an isomorphism 
\begin{equation*}
\psi^\ast:\G_1^0\overset{\cong}{\to}\Stab(i)\subset \G_n
\end{equation*}
where $i$ is the element of $[n]$ not collapsed by $\phi$.
\end{lem}

\begin{proof}
It is clear that $\psi^\ast(\G_1^0)\subset\Stab(n)$. To see that it is an isomorphism, it suffices to prove it is an isomorphism for $\phi\in\Delta$, since $\Stab(i)$ is related to $\Stab(0)$ and $\Stab(n)$ by conjugation by an element in $\G_n$ for any $i$. 

In this case, we notice that TFDC:
\begin{equation*}
\xymatrix{
 & [n]\ar[dr]^\phi & \\
 [0]\ar[ur]^{i_n}\ar[rr]_{j_1} & & [1]
}
\end{equation*} 
where $j=\lambda(\phi)(i)$. This means that we have:
\begin{equation*}
i_n^\ast \circ \phi^\ast=j_1^\ast
\end{equation*}
And, since two of these maps are isomorphisms, so is the third. 
\end{proof}

\begin{exmp}[Non-examples]
The trivial crossed simplicial group $\Delta$ is not balanced, since 
\begin{equation*}
\Hom_\Delta([0],[n])\not\cong\Hom_\Delta([n],[0])
\end{equation*}
for $n\geq 1$.

Similarly, the braid crossed simplicial group $\Delta B$ is not balanced. To see this, we notice that, if it were balanced, we would have that
\begin{eqnarray*}
B_{n+1} & \cong &\Hom_{\Delta B}([0],[n])\cong\Hom_{\Delta} ([0],[n]) 
\end{eqnarray*}
which contradicts the infinitetude of $B_{n+1}$ for $n>0$. 
\end{exmp}

\begin{exmp}\label{ex:BH}
We already have shown that planar crossed simplicial groups provide examples of balanced crossed simplicial groups. However, we can also construct new examples of balanced crossed simplicial groups from old ones. 

Let $H$ be a group, and $BH$ be the associated groupoid, and let $\Delta\G$ be a balanced crossed simplicial group. We can form the product
\begin{equation*}
\Delta\G\times BH
\end{equation*}
to get a new crossed simplicial group (Canonical factorization is immediate from definitions) $\Delta\mathfrak{GH}$. We have that
\begin{equation*}
\mathfrak{GH}_n=\G_n\times H
\end{equation*}
Since both $BH$ and $\Delta\G$ are self-dual, we can construct a duality on $\Delta\mathfrak{GH}$. And, since a pullback morphisms are simply the identity on $H$, we see that the remaining properties are satisfied. 
\end{exmp}

We can make some more sense of the condition that a crossed simplicial group be balanced by relating out definition to the elegant classification given in \cite{FL}.

\begin{defn}
The \emph{Weyl Crossed Simplicial Group $\Delta\W$} is a crossed simplicial group with automorphism groups 
\begin{equation*}
\W_n=\Z/2\Z \wr S_{n+1}
\end{equation*}
More explicitly, it is the category whose objects are the sets $\{0,1\ldots n\}$, and whose morphisms are maps of sets together with signed linear orders on fibers. (See \cite{DK} section 1.2 for more details).
\end{defn}

\begin{prop}
For $\Delta\G$ a crossed simplicial group
\begin{enumerate}
\item There is a canonical functor $\pi:\Delta\G\to\Delta\W$ preserving $\Delta$.
\item There is a sequence of functors, unique up to isomorphism of crossed simplicial groups
\begin{equation*}
\Delta\G^\prime \to \Delta\G \overset{\pi}{\to} \Delta\G^{\prime\prime}
\end{equation*}
Such that the induced sequences of automorphism groups are all short exact. $\Delta\G^{\prime\prime}$ is the image of $\pi$ (a crossed simplicial subgroup of $\Delta\W$), and $\Delta\G^\prime$ is a simplicial group.  
\end{enumerate}
\end{prop}
\begin{proof}
\cite{FL} proposition 3.16 or \cite{DK} theorem 1.7
\end{proof}

As a result of the proposition, we see that there is a classification of crossed simplicial groups: they are all extensions of crossed simplicial subgroups of the Weyl group by a simplicial group. We can list these subgroups, as in figure \ref{tab:csgclass}.

\begin{figure}[h]
\begin{tabular}{l | c }
\hline 
Name & Subgroup of $\Delta\W$\\
\hline 
trivial & $\{1\}$ \\
reflexive & $\{\Z/2\Z\}$ \\
cyclic & $\Lambda$ \\
dihedral & $\Xi$\\
Symmetric & $\{S_{n+1}\}$\\
Reflexosymmetric & $\{\Z/2\Z \ltimes S_{n+1}\}$\\
Weyl & $\{ \Z/2\Z \wr S_{n+1}\}$ \\
\end{tabular}
\caption{Crossed Simplicial subgroups of $\Delta\W$.}
\label{tab:csgclass}
\end{figure}

\begin{rem}
Note that the canonical functor $\pi$, in the case of all planar crossed simplicial groups, has image $\Lambda$ or $\Xi$, and kernel a constant simplicial group. See, for instance, the example of the quaternionic crossed simplicial group as worked out in \cite{FL}. This holds more generally, as we can see in the next lemma.
\end{rem}

\begin{lem}\label{lem:ballem}
Every balanced crossed simplicial group is an extension of $\Lambda$ or $\Xi$ by a constant simplicial group. 
\end{lem}

\begin{proof}
We can consider the injective homomorphisms 
\begin{equation*}
\omega_n^\ast:\G_0 \to \G_n
\end{equation*}
defined by the maps $\omega_n: [n]\to[0]$ in $\Delta$.  Given $k,h\in\G_n$, the cosets of $\omega_n^\ast \G_0$ in $\G_n$ associated to $h$ and $k$ will be the same if and only if, for every $g\in\G_0$, there is an $m\in\G_0$ such that 
\begin{equation*}
(\omega_n^\ast g) k=(\omega_n^\ast m)h
\end{equation*}
that is, if and only if
\begin{equation*}
g\circ\omega_n\circ k=m\circ \omega_n \circ h
\end{equation*}
By canonical factorization $\omega_n^\vee=\gamma\circ a$, and so, applying duality, we have the condition:
\begin{equation*}
k^\vee \circ \gamma\circ a\circ g^\vee=h^\vee \circ \gamma\circ a\circ m^\vee
\end{equation*}
or
\begin{equation*}
(k^\vee)^\ast(\gamma)\circ\gamma^\ast(k^\vee)\circ a\circ g^\vee=(h^\vee)^\ast(\gamma)\gamma^\ast(h^\vee) \circ a \circ m^\vee 
\end{equation*}
So we see that the two cosets will be the same if and only if we have that
\begin{equation*}
(k^\vee)^\ast (\gamma)=(h^\vee)^\ast(\gamma)
\end{equation*}
So that our cosets are in one-to-one correspondence with elements of $\Hom_\Delta([0],[n])$. Therefore, we have that the duality places a `linear growth' condition on the order of the automorphism groups:
\begin{equation*}
\left| \frac{\G_n}{\omega_n^\ast \G_0} \right|=n+1
\end{equation*}
Since we need the functor $\pi$ from the theorem to be surjective, this means that $\Delta\G$ can only be an extension of the trivial, reflexive, dihedral, or cyclic crossed simplicial groups.

In a simplicial group, the group elements act trivially on the morphisms in $\Delta$, so that, in the extension sequence for $\Delta\G$
\begin{equation*}
\Delta\G^\prime\to\Delta\G\to\Delta\G^{\prime\prime}
\end{equation*}
we have that, in $\Delta\G^\prime$
\begin{equation*}
\Stab(i)=\G_n^\prime
\end{equation*}
Now, the condition that $i_n^\ast:\Stab(n)\to\G_0$ be an isomorphism means that 
\begin{equation*}
i_n^\ast:\G_n^\prime\to\G_0^\prime
\end{equation*}
must be an isomorphism. This tells us that $\omega_n^\ast$ is also an isomorphism for all $n$, and so $\Delta\G^\prime$ must be constant.  Now, note that an extension of the trivial or reflexive groups by a constant simplicial group must itself be semi-constant, hence it cannot be equipped with a duality. Therefore, the only cases left are extensions of $\Lambda$ and $\Xi$ by constant simplicial groups.
\end{proof}

\begin{rem}
While a more complete classification of balanced crossed simplicial groups in terms of the Fiedorowicz-Loday exact sequence could provide a clear topological interpretation of the meaning of a `crossed simplicial field theory' (once such a notion has been defined), lemma \ref{lem:ballem} already provides us with some insight into a topological meaning of balanced-ness. 

As we will see in the next section, graphs structured over a planar crossed simplicial group correspond to bordisms equipped with a certain reduction of the structure group. In particular, graphs structured over $\Lambda$ correspond to oriented bordisms, and graphs structured over $\Xi$ correspond to unoriented bordisms. From the lemma we can see that, given a balanced crossed simplicial group $\Delta\G$, a $\Delta\G$-structure on a graph $\Gamma$ amounts to a pair of lifts
\begin{equation*}
\xymatrix{
 & \mathcal{G}\ar[d]\\
 & \mathcal{K}\ar[d] \\
 I(\Gamma)\ar[r]\ar[ur]\ar[uur] & \mathbf{FSet}
}
\end{equation*}
of the incidence diagram (where $\mathcal{K}$ is either $\Lambda$ or $\Xi$). Since the intermediate lift has a topological significance, the lemma tells us that, in some sense, graphs structured over a balanced crossed simplicial group correspond to oriented or unoriented bordisms, together with some additional structure. In all the worked examples, this additional structure appears to, itself, be topological in nature. We might therefore loosely conjecture that every balanced crossed simplicial group corresponds to a topologically defined bordism category.
\end{rem}
 
\section{The $G$-Structured Cobordism Category}

We are interested in topological field theories on a very specific bordism category. Before defining that, however, we recall some more basic constructions.

\subsection{$2\text{Cob}$}
We here follow \cite{K} in giving a brief description of the 2-dimensional oriented cobordism category. 

\begin{defn}
A \emph{strict cobordism} from $S_0$ to $S_1$ (closed manifolds of dimension $n-1$) is a manifold with boundary $(M,\partial M)$ of dimension $n$ along with inclusion maps 
\begin{equation*}
f: S_0 \rightarrow M \leftarrow S_1: g
\end{equation*} 
such that 
\begin{equation*}
(f,g):S_0\sqcup S_1 \overset{\cong}{\rightarrow} \partial M
\end{equation*}
is a diffeomorphism. We say two strict cobordisms are equivalent if there is a diffeomorphism $M\to M_\prime$ fixing the boundary such that TFDC
\begin{equation*}
\xymatrix{
 & M^\prime & \\
 S_0\ar[dr]\ar[ur] & & S_1\ar[dl]\ar[ul] \\
 & M\ar[uu]^{\cong} &
}
\end{equation*}
\end{defn}

To glue strict cobordisms, we note that, by the regular interval theorem (\cite{K} 1.2.3), for any strict cobordism $(M,S_0,S_1)$, we can find a collar neighborhood $S_0\times [0,1]$ in $M$ (and similarly for $S_1$). We can then glue the collar neighborhoods by the rule that 
\begin{equation*}
S_0\times[0,1] S_0\times [0,1]\xymatrix@C=1.5em{{}\ar@{~>}[r]&{}} S_0\times [0,2]
\end{equation*}  

The equivalence class of the strict cobordism created by such a gluing is uniquely defined. More precisely (see \cite{K} 1.2.1) 
\begin{thm}\label{thm:smoothcob}
Let $(M_1,S_0,S_1)$ and $(M_2,S_1,S_2)$ be two strict cobordisms. There exists a smooth structure on $M_1M_2:=M_1\bigsqcup_{S_1} M_2$ such that $M_1\to M_1M_2$ and $M_2\to M_1M_2$ are diffeomorphisms onto their image. This smooth structure is unique up to diffeomorphism fixing $(S_0,S_1,S_2)$. 
\end{thm}

\begin{defn}
Given an oriented strict cobordism $(M,S_0,S_1)$, we call $S_i$ an \emph{in-boundary} if its orientation matches that induced by the orientation on $M$, and an \emph{out-boundary} otherwise. In such a case, we only allow gluings out-boundary to in-boundary, so that there is a unique orientation defined on the composition.
\end{defn}

\begin{defn}
The category $2\text{Cob}^{cl}$, the \emph{closed two-dimensional oriented cobordism category}, has objects given by disjoint unions of (oriented) circles, and morphisms given by equivalence classes of oriented strict cobordisms. 
\end{defn}

This definition, however, is quite strict in some of it's requirements. We can relax the definition a little to get a more useful notion. 

\begin{defn}
A \emph{cobordism} $(M,S_0,S_1)$ is defined in precisely the same way as a strict cobordism, except that $S_0$ and $S_1$ are no longer expected to exhaust the boundary of $M$, but instead, be embedded into it. Gluing and orientation still operate in precisely the same way, and do the definitions of the in- and out-boundaries. If $\partial M_{in}$ and $\partial M_{out}$ denote the in- and out-boundaries respectively, we call $\partial M_{free}:=\partial M\setminus (\partial M_{in}\sqcup \partial M_{out})$ the \emph{free boundary} of the cobordism. 
\end{defn}

\begin{rem}
There is an version of Theorem \ref{thm:smoothcob} for cobordisms. 
\end{rem}

Now, we can modify our definition somewhat (for more details, see, eg, \cite{B} definition 1.1).

\begin{defn}
The category $2\text{Cob}$, the \emph{two-dimensional oriented cobordism category}, is defined to have objects closed oriented 1-manifolds (disjoint unions of circles and intervals), and morphisms given by equivalence classes of oriented cobordisms. $2\text{Cob}^{cl}$ is a (not full) subcategory of $2\text{Cob}$. 

There is another important subcategory of $2\text{Cob}$, the \emph{open cobordism category} $2\text{Cob}^o$, which is the \textbf{full} subcategory whose objects are disjoint unions of intervals. 
\end{defn}

\begin{rem}
By a historical accident of linguistics and notation, the terms cobordism and bordism have come to be used for the same objects. In this paper, we will favor the term cobordism. 
\end{rem}

\begin{rem}
The categories above are examples of a more general notion of \emph{cobordism category}, for which we will not give an exact definition. In general, we will use the term to refer to any category whose morphisms are defined as equivalence classes of cobordisms of some sort (possibly with additional structure, as we will see shortly).  
\end{rem}

\subsection{$\Cob$ and $\Cor$}

We now turn our attention to cobordism categories with special relevance to the notions developed in the first 3 sections.

\begin{defn}
For a planar Lie Group $G$, a \emph{$\mathbb{G}$-structured surface} is a manifold (possibly with boundary) $M$ equipped with a $\mathbb{G}$-principal bundle $F_\mathbb{G}\to M$ and a reduction of the structure group equivariant along $p_\mathbb{G}: \mathbb{G}\to GL(2,\R)$
\begin{equation*}
\pi:F_\mathbb{G}\to Fr_M
\end{equation*} 

A \emph{structured diffeomorphism} 
\begin{equation*}
(\phi,\tilde{\phi}): (M,F_\mathbb{G}, \pi)\to (M^\prime,F_\mathbb{G}^\prime,\pi^\prime)
\end{equation*}
is a pair where $\phi:M\to M$ is a diffeomorphism, and $\tilde{\phi}:F_\mathbb{G} \to F_\mathbb{G}^\prime$ is a $\mathbb{G}$-equivariant diffeomorphism such that TFDC
\begin{equation*}
\xymatrix{
F_\mathbb{G}\ar[d]_\pi \ar[r]^{\tilde{\phi}} & F_\mathbb{G}^\prime\ar[d]^{\pi^\prime} \\
Fr_M \ar[r]_{d\phi} & Fr_{M^\prime}
}
\end{equation*}
\end{defn}

\begin{exmps} 
\begin{enumerate}
\item An $GL^+(2,\R)$-structured surface is simply an oriented surface.
\item An $GL(2,\R)$-structured surface has no additional data.
\item A $\tilde{GL}^{+}(2,\R)$-structured surface is an oriented surface together with an chosen isotopy class of trivializations of the frame bundle (see \cite{DK} examples 3.3 for more details).
\end{enumerate}
\end{exmps}

\begin{defn}
For a planar lie group $\mathbb{G}$, the \emph{$\mathbb{G}$-structured interval} $\scri$ is the interval $I=[0,1]$ equipped with the reduction of structure group
\begin{equation*}
\mathbb{G}\times I\to Fr(TI\oplus\R)
\end{equation*} 
Where the map in question is given by the homomorphism
\begin{equation*}
\mathbb{G}\times I\to GL(2,\R)\times I
\end{equation*}
and the identification $GL(2,\R)\cong Fr(TI\oplus\R)$ given by choosing the frame $\partial_x\in TI$ (where $x$ is the coordinate on $I$) and $1\in\R$. 

A \emph{structured embedding of $\scri$} into a structured surface $(M,F_\mathbb{G})$ is data $(f,\hat{f},\tilde{f})$  where $f:I\to M$ is a $C^\infty$ embedding, $\hat{f}:Fr(TI\oplus\R)\to Fr_M$ extends $df$, and $\tilde{f}:\mathbb{G}\otimes I\to F_\mathbb{G}$ is a bundle map covering $\hat{f}$.  
\end{defn}

\begin{defn}
A \emph{$G$-structured cobordism} $(M,S_0,S_1)$ is a $G$-structured surface with boundary $(M,F_\mathbb{G})$, with
\begin{equation}
S_i=\bigsqcup_{\ell_i} \scri
\end{equation}
Equipped with structured embeddings 
\begin{equation*}
(f,\tilde{f}):S_0\to (M,F_\mathbb{G}) \leftarrow S_1:(g,\tilde{g})
\end{equation*}
with image in $\partial M$. 

Two such structured cobordisms $M_1$ and $M_2$ are considered equivalent if there is a structured diffeomorphism $(\phi,\tilde{\phi})$ fixing the boundary such that TFDC
\begin{equation*}
\xymatrix{
 & M_2 & \\
 S_0\ar[dr]\ar[ur] & & S_1\ar[dl]\ar[ul] \\
 & M_1\ar[uu]_{(\phi,\tilde{\phi})} &
}
\end{equation*}
\end{defn}

\begin{rem}
Notice that, in the case of a circle $S^1$ embedded in the boundary of $M$, there is the possibility of the datum of a reduction of the structure group
\begin{equation*}
F_{S^1}\to Fr(TS^1\oplus\R)
\end{equation*} 
being non-trivial (see, for example, \cite{NR}, 2.4). However, in the case of an embedded interval $I$, contractability allows us to argue that this additional structure is trivial. This will be important in the construction of the cobordism category.
\end{rem}

\begin{rem}
When discussing in- and out- boundaries for $\mathbb{G}$-structured cobordisms, there is a subtle distinction to be made between two cases. In general, we refer to $S_0$ as the in-boundary and $S_1$ as the out-boundary. However, if the map $\mathbb{G}\to GL(2,\R)$ factors through $GL^+(2,\R)$, then the choice of a reduction of structure group includes a choice of orientation. Moreover, the choice of structured embedding defines an outward normal by the image of $1\in\R$ under $\hat{f}$. We require that, for $S_0$ this point inwards, and for $S_1$, this point outwards. 
\end{rem}

\subsubsection{Gluing Structured Cobordisms}

We first consider the cylinder over a structured boundary component
\begin{equation*}
S=\bigsqcup_{\ell_i} \scri
\end{equation*}
That is, the cobordism given by 
\begin{eqnarray*}
C & = & S\times [0,1]\\
\pi_C: \mathbb{G}\times S\times [0,1] & \to & Fr(TC)\\
\end{eqnarray*}
With the obvious boundary embeddings

We can glue two such cylinders over $(S,_SF)$ $C_1$ and $C_2$ simply by extending the interval to [0,2]. This yields, as above, a cobordism $C_1C_2:=S\times[0,2]$ such that $C_i\to C_1C_2$ is a structured diffeomorphism onto its image. 

Now, given two more general structured cobordisms $(M_1,S_0,S_1)$ and $(M_2,S_1,S_2)$, we try to find a "trivializing collar neighborhood", that is, a neighborhood of $S_1$ and a structured diffeomorphism to the cylinder over $S_1$. 
\begin{lem}
There is a neighborhood of $S_1$ in $M_i$ that is equivalent to the structured cylinder over $S_1$ via structured diffeomorphism. 
\end{lem}   
\begin{proof}
We prove the lemma for a single boundary interval. Take a smooth collar neighborhood $C$ of $I$ in $M$. Since $C$ is contractible, we can trivialize $F_\mathbb{G}$ over $C$, and can identify this trivialization with the one on $I$ defined by the structured embedding of $\scri$.
\end{proof}

This allows us to define a gluing of the two cobordisms, using the gluings for cylinders. 

\begin{prop}
Let $(M_1,S_0,S_1)$ and $(M_2,S_1,S_2)$ be two structured cobordisms. There exists a $\mathbb{G}$-structure on $M_1M_2:=M_1\bigsqcup_{S_1} M_2$ such that $M_1\to M_1M_2$ and $M_2\to M_1M_2$ are structured diffeomorphisms onto their image. This structure is unique up to structured diffeomorphism fixing $(S_0,S_1,S_2)$. 
\end{prop}
\begin{proof}
Firstly, we know that one such structure $F_{cyl}$ exists. Now let $(M_1M_2,F^\prime)$ be a second such structure. We can apply theorem \ref{thm:smoothcob} to $F_{M_2}$ and $F_{M_1}$ to see that there is a diffeomorphism $\tilde{\phi}:F^\prime \to F_{cyl}$ commuting with the inclusions of the bundles $F_{M_2}$ and $F_{M_1}$ (note that we can only apply the theorem because of the explicit choice of trivialization of the bundle over the embedded boundary components implicit in our choice of structured embedding, which gives us a canonical identification of the boundaries of the total spaces). This descends to a diffeomorphism $\phi:M_1M_2 \to M_1M_2$. 

The only thing, then, remaining to be checked, is $\mathbb{G}$-equivariance. However, this is clear since $\tilde{\phi}$ must commute with the structured inclusion of $M_1$ and $M_2$. 
\end{proof}

\begin{defn}
The \emph{$\mathbb{G}$-structured cobordism category} $\Cob$ has objects closed structured 1-manifolds, and morphisms equivalence classes of $\mathbb{G}$-structured cobordisms equipped with linear orderings of the inputs and outputs. 
\end{defn}

\begin{rem}
While it has not been discussed above, all of the cobordism categories defined above have the obvious monoidal structure (disjoint union), making them into symmetric monoidal categories. 
\end{rem}

\begin{rem}
There is another way to think about the various boundary components of a cobordism. First note that, for an open structured cobordism $(M, S_0,S_1)$, we can think of $M$ as a surface $\Sigma$ with boundary punctured at a set of interior marked points corresponding to the components of the boundary that contain no images of points in $S_0$ or $S_1$. We will call this set of marked points $P^\circ$. Additionally, for a boundary component $A$ that does contain intervals $\{I_j\}_j$ from $S_0$ or $S_1$, we can define a set of marked points for $A$ by choosing one point in each component of $A\setminus \bigcup_j I_j$. We call the set of these points $P^\partial$.  

From this construction, we can represent our cobordism $(M,S_0,S_1)$ as a marked surface $(\Sigma,P)$ where $P=P^\circ\cup P^\partial$ , equipped with a $\mathbb{G}$-structure on $\Sigma\setminus P$. Up to structured diffeomorphism, we can retrieve $(M,S_0,S_1)$ from this data. 

We will in general refer to the marked points $P^\circ$ in the interior as \emph{punctures} of $\Sigma$. 
\end{rem}

There is a connection between $\Cob$ and the more algebraic/combinatorial formalism developed in previous sections. This comes via the use of an intermediate category $\Cor$.

\begin{defn}
The \emph{structured corolla category} $\Cor$ is the symmetric monoidal category whose objects are disjoint unions of an object $\mathscr{I}$ (including the empty union), and in which a morphism
\begin{equation*}
\mathscr{I}^{\otimes n} \to \mathscr{I}^{\otimes m}
\end{equation*}
is given by an equivalence class of $\Delta\G$-structured augmented graphs with $n$ incoming half-edges and $m$ outgoing half-edges, equipped with an ordering of the incoming and of the outgoing half-edges.
\end{defn}

\begin{rem}
Note that the concept of a $\Delta\G$-structured graph with only one vertex and no half-edges is meaningless under the definitions given.  
\end{rem}

\subsection{Topological Field Theories}

\begin{defn}
A \emph{Topological Field Theory (TFT)} is a symmetric monoidal functor from a cobordism category $\text{Cob}$ to an abelian category $\mathcal{C}$. 
\begin{equation*}
Z: \text{Cob} \to \mathcal{C}
\end{equation*}
For our purposes, we will consider functors to $\text{Vect}_k$ for some field $k$.
\end{defn}

\begin{exmp}
It is folkloric that TFT's on $2\text{Cob}^o$ are frobenius algebras over $k$.
\end{exmp}

\begin{defn}
Let $\GOTFT_k$ ($G$-structured open TFT over $k$) 
be the category 
\begin{equation*}
\GOTFT_k=\text{Fun}^\otimes(\Cob,\Vect_k)
\end{equation*}
whose objects are topological field theories:
\begin{equation*}
Z: \Cob \to \text{Vect}_k
\end{equation*}
and whose morphisms are natural transformations.
\end{defn}

\begin{defn}
Let $\CSFT_k$ (crossed simplicial field theory over $k$) 
be the category
\begin{equation*}
\CSFT_k=\text{Fun}^\otimes (\Cor,\Vect_k)
\end{equation*}
whose objects are topological field theories
\begin{equation*}
Z:\Cor\to \Vect_k
\end{equation*}
and whose morphisms are natural transformations.
\end{defn}

\begin{rem}
In the construction of $\Cob$, we implicitly included as morphisms equivalence classes of structured surfaces without any boundary components. However, within the framework of $\text{GOTFT}_k$, these morphisms correspond to choices of field automorphisms  (multiplication by elements in the field). Their compositions only yield more morphisms without boundary components, and they cannot be realized as the composition of morphisms with boundary components. Since, for any field $k$, one can arbitrarily assign field elements to connected closed manifolds without boundary, these morphisms are not relevant to a classification of GOTFT's. For the remainder of this paper, including the following theorem, we will assume that these morphisms are not included in $\Cob$.  
\end{rem}

\begin{thm}\label{thm:cobequiv}
There is an equivalence of categories $\Cob \cong \Cor$.
\end{thm}

Before we prove this theorem, we need some concepts and notation.

\begin{defn}
A \emph{simple curve} on a marked surface $(\Sigma,P)$ is a continuous map $\gamma:[0,1]\to\Sigma$ such that 
\begin{itemize}
\item The endpoints of $\gamma$ are in $P$
\item $\gamma$ is non-self-intersecting
\item If the endpoints of $\gamma$ are the same point $q$, then $\gamma\neq0\in \pi_1 \left((\Sigma\setminus P)\cup q,q\right)$
\end{itemize}
An \emph{arc} is an equivalence class of simple curves in $\Sigma$ under isotopy and reversal of parameterization. Two arcs are said to be \emph{compatible} if they can be represented by non-intersecting curves. 
\end{defn}

\begin{defn}
An \emph{ideal triangulation} of a marked surface $(\Sigma,P)$ is a maximal collection of pairwise compatible arcs.

A \emph{spanning graph} for $(\Sigma,P)$ is a graph $\Gamma$ embedded into $\Sigma\setminus P$ such that $\partial\Gamma\subset\partial \Sigma$ and such that the maps
\begin{eqnarray*}
\Gamma &\to & \Sigma\setminus P\\
\partial\Gamma & \to & \partial\Sigma\setminus P
\end{eqnarray*}
are homotopy equivalences. A spanning graph is called 3-valent if all of its vertices have valency 1 or 3.
\end{defn}

For much of what follows, we will need to work in a more specific case for our marked surfaces.

\begin{defn}
Let $(\Sigma,P)$ be a marked $C^\infty$ surface with boundary. Denote $\Sigma^\circ=\Sigma\setminus\partial\Sigma$. The \emph{Schottky double} $\Sigma^{\#}$ of $\Sigma$ is obtained by taking the orientation cover $\widetilde{\Sigma^\circ}$ then compactifying it by gluing in a single copy  of $\partial\Sigma$. This yields a two-sheeted covering $\pi:\Sigma^\# \to \Sigma$ ramified along the boundary. Additionally, We can equip $\Sigma^\#$ with the structure of a marked surface by taking $P^\#=\pi^{-1}(P)$.
\end{defn}

\begin{defn}
A marked surface $(\Sigma,P)$ is called \emph{stable} if
\begin{enumerate}
\item $P\neq\emptyset$ and $P$ meets every boundary component of $\Sigma$.
\item $\eta(\Sigma^\# \setminus P^\#)<0$
\end{enumerate}
\end{defn}

\begin{rem}
To clarify this definition somewhat, we note that in the oriented case,  the second condition amounts to requiring that $(\Sigma,P)$ is not
\begin{itemize}
\item $S^2$ with $\vert P\vert\leq 2$
\item $D^2$ with $\vert P\vert=1$
\item $D^2$ with $\vert P\vert=2$ and $P\subset \partial \Sigma$.
\end{itemize}

In the unoriented case, we additionally prohibit:

\begin{itemize}
\item $\R P^2$ with $\vert P\vert = 1$
\end{itemize}
\end{rem}

We then have the following (see, eg, \cite{DK2} for the oriented case)

\begin{prop}\label{prop:spangraph}
Let $(\Sigma,P)$ be a stable marked surface. Then taking the dual graph creates a bijection between the set of ideal triangulations  of $(\Sigma,P)$ and the set of isotopy classes of 3-valent spanning graphs of $(\Sigma,P)$.
\end{prop}

Additionally, we can notice that, as a particular case of the construction of a $\G$-structure on a graph embedded in a $\mathbb{G}$-structured surface $M$, if we have an edge $e$ leading to a point $x$ on the boundary, it's endpoint can be made a $\G$-structured set by taking as our $\G_1$-torsor $\mathcal{O}_x$ the preimage of the germs defined by the edge viewed as elements of $C(T_x M)$ under our $\Homeo^G(S^1)$-bundle. Given a spanning graph $\Gamma$ embedded in the surface, we can define an "augmentation map" by noticing that, WLOG, we can take the point $m(e)$ defining the edge $e$ connected to $x$ in $\Gamma$ to be contained in a trivializing open set containing $x$, thereby canonically identifying the torsors defined for $x$ and $m(e)$. 

This does not quite define an augmentation, since we still have not chosen an identification of $(C(T_x M), J, \mathcal{O}_x)$ with $[1]$. However, the trivialization of the $\mathbb{G}$-bundle $F_\mathbb{G}\to M$ on the boundary gives us a trivialization of the $\Homeo^G(S^1)$ bundle
\begin{equation*}
\xymatrix{
F_\mathbb{G}\times_\mathbb{G}\Homeo^G(S^1)\ar[d]_\pi & & \\
Fr_M \times_{GL(2,\R)}\Homeo(S^1) \ar[rr]^{=} & & \Homeo(S^1,C(TM)
}
\end{equation*} 
on the boundary. The identity element in this trivialization projects down to the unit normal $v\in C(T_x M)$ defined by the image of $1\in\R$ under the embedding of $\scri$ into $M$, and so, taking WLOG $(v,-v)$ 
to be the germs in $C(T_x M)$ defined by the edge attached to $x$, the trivializing element of $F_\mathbb{G}\times_\mathbb{G}\Homeo^G(S^1)$ defines a trivialization of the torsor $\mathcal{O}_x$. 

\begin{lem}
The augmentations and $\G$-structure on $\Gamma$ defined by the process above are invariant under equivalence of cobordisms
\end{lem}

\begin{proof}
A structured diffeomorphism $(\phi,\tilde{\phi})$ yields maps 
\begin{equation*}
\xymatrix{
F_\mathbb{G}\times_\mathbb{G}\Homeo^G(S^1)\ar[d]_\pi \ar[rr]^{\tilde{\phi}} & &F_\mathbb{G}\times_\mathbb{G}\Homeo^G(S^1)\ar[d] \\
\Homeo(S^1,C(TM)\ar[rr]^{\circ d\phi} & & \Homeo(S^1,C(TM)
}
\end{equation*}
Which give isomorphisms of structured circles for every point $y\in M$. By Lemma 4.26 in \cite{DK}, this gives the desired equivalence of $\G$-structures. However, since the structured diffeomorphism must commute with the inclusions of $\scri$, the induced map of structured circles is the identity on the boundary (augmentation) elements under the identifications with $[1]$ given by the inclusions of $\scri$.
\end{proof}

And, conversely:

\begin{lem}\label{lem:structrecon}
Given $\G$-structured augmented spanning graph $\Gamma$ embedded in a cobordism $M$, there is an induced $\mathbb{G}$-structure $F_\mathbb{G}$ on $M$, and trivializations of $F_\mathbb{G}$ over the embedded boundary intervals defined by $\partial\Gamma$. These structures are unique up to structured diffeomorphism. 
\end{lem}

\begin{proof}
The first part of the lemma is again covered by Lemma 4.26 in \cite{DK}. To find trivializations, first note that, for a point $x\in\partial\Gamma$ with germs $(v,-v)\in C(T_x M)$ the augmentation defines a trivialization of the $\G_1$ torsor $\pi^{-1}(v,-v)$, and the identity element in this trivialization gives an element $d$ which defines a trivialization of the $\Homeo^G(S^1)$-bundle $F_\mathbb{G}\times_\mathbb{G}\Homeo^G(S^1)$.

The map $f:\mathbb{G}\to\Homeo^G(S^1)$ has a homotopy inverse $g$. We can define a trivialization of $F_\mathbb{G}$ by any element $z\in\pi^{-1}(d)$ where $\pi:F_\mathbb{G}\to F_\mathbb{G}\times_\mathbb{G}\Homeo^G(S^1)$. To see that this yields a well-defined cobordism, notice the homotopy $H_t$ from $g\circ f$ to $id_\mathbb{G}$ gives rise to a path from $y$ to $z$ for any $y,z\in\pi^{-1}(d)$. Choosing a collar neighborhood of the embedded interval in question, and parameterizing it by $I\times I$, this path gives rise to a structured diffeomorphism which moves from the trivialization defined by $y$ to that defined by $z$ on the boundary, and reduces to the identity on the interior. 
\end{proof}

It is clear from the definitions that the processes in the two lemmas are inverse to one another, so we can now return to Theorem \ref{thm:cobequiv}.

\begin{proof}[Proof of Theorem \ref{thm:cobequiv}]
Fix a crossed simplicial group $\Delta\G$ corresponding to the planar Lie group $G$. We want to construct a functor 
\begin{equation*}
T: \Cob\to\Cor
\end{equation*}
the value of which on objects is obvious. 

If we restrict ourselves to stable marked surfaces $(\Sigma,P)$, we can give this functor's value by taking an ideal triangulation of $(\Sigma,P)$, and taking the dual graph $\Gamma$ of this triangulation. By \ref{prop:structgraphprop}, $\Gamma$ is endowed with a canonical $\Delta\G$-structure. It is a classical fact that any two such triangulations of $(\Sigma,P)$ will be be related by the 2-2 Pachner move, as pictured in figure \ref{fig:pachmove}. This means that any two structured graphs obtained through this method will be related by edge contractions, so that the equivalence class of $\Delta\G$-structured graphs obtained through this procedure is well-defined. 
\begin{figure}[h]
\begin{tikzpicture}
\draw [thick] (0,0) -- (0,2);
\draw [thick] (0,0) -- (2,0);
\draw [thick] (0,2) -- (2,2);
\draw [thick] (2,0) -- (2,2);
\draw [thick] (0,0) -- (2,2);
\draw [thin,<->] (2.5,1) -- (3.5,1);
\draw [thick] (4,0) -- (4,2);
\draw [thick] (4,0) -- (6,0);
\draw [thick] (4,2) -- (6,2);
\draw [thick] (6,0) -- (6,2);
\draw [thick] (6,0) -- (4,2);
\draw [thin, red] (1.5,-0.5)--(1.5,0.5)--(2.5,0.5);
\draw [thin, red] (0.5,2.5)--(0.5,1.5)--(-0.5,1.5);
\draw [thin, red] (1.5,0.5)--(0.5,1.5);
\draw [thin, red] (3.5,0.5)--(4.5,0.5)--(4.5,-0.5);
\draw [thin, red] (5.5,2.5)--(5.5,1.5)--(6.5,1.5);
\draw [thin, red] (4.5,0.5)--(5.5,1.5);
\path [fill=red] (0.5,1.5) circle [radius=0.05];
\path [fill=red] (1.5,0.5) circle [radius=0.05];
\path [fill=red] (4.5,0.5) circle [radius=0.05];
\path [fill=red] (5.5,1.5) circle [radius=0.05];
\end{tikzpicture}
\caption{The 2-2 Pachner Move, with the dual graph of the triangulation marked in red.} \label{fig:pachmove}
\end{figure}
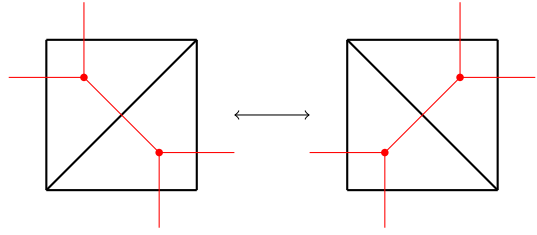

From the theory of ribbon/m\"obius graphs (see, eg \cite{B}) we know that the underlying $C^\infty$ marked surface $(\Sigma,P)$ can be reconstructed from an equivalence class of ribbon/m\"obius graphs. Every planar crossed simplicial group admits a canonical functor $\Delta\G\to\Xi$ or $\Delta\G\to\Lambda$ which in some sense forgets the extra structure. Hence, by passing through this functor, and then performing the `thickening' procedure, we can retrieve the $C^\infty$ marked surface from its image under the functor $T$ as defined above. 

Now, we need to confirm that we can retrieve the $\mathbb{G}$-structure on $\Sigma$ from our equivalence class of graphs. However, lemma 4.26 in \cite{DK} tells us precisely this. Hence, the functor we have defined on the restriction of $\Cob$ to stable surfaces is faithful (presuming, of course, that the assignment defined above is a functor, which we will prove shortly). Moreover, we can immediately see that it is full by looking at the analogous cases where we extract m\"obius or ribbon graphs from the surface. Actually, there is a subtlety here. We have not yet showed that the trivializations over embedded boundary components can be reconstructed, so the proof of fully faithfulness is incomplete. However, this amounts to precisely the statement of lemma \ref{lem:structrecon}.

Before proving functorality, we want to define $T$ in the remaining (non-stable) cases. We define $\Delta\G$-structured graphs by embedding graphs in the surfaces in question, and then taking the induced $\Delta\G$ structure:
\begin{itemize}
\item For $D^2$ with a single marked point, the image of $T$ will be the equivalence class defined by the graph with a single vertex embedded in the disk, and a single half-edge leading to the boundary of the disk.
\item For $D^2$ with two marked points on the boundary, we again define the image by an embedded graph. The marked points break the boundary into two sections. The image under $T$ is defined to be the graph with a single vertex embedded in the center of the disk, and a half-edge leading to each section of the boundary.
\item For $S^2$ with two marked points, the image of $T$ is the graph with a single vertex and a loop, embedded so that the loop generates the fundamental group of $S^2\setminus P$.
\item For $S^2$ with one marked point, the image of $T$ is the graph with two embedded vertices, and an edge between them.
\item For $\R P^2$ (where applicable) with one marked point, the image of $T$ is the graph with a single vertex and a loop embedded so that the loop generates the fundamental group of $\R P^2 \setminus P$.  
\end{itemize}

It remains, then, to show functorality and finish proving fully faithfulness. To see functorality in the case of stable structured cobordisms, we notice that composing two such cobordisms concatenates their spanning graphs, and yields a spanning graph of the composed cobordism. Since such spanning graphs are in bijection with triangulations of the cobordism, we see that the concatenated graph can be realized as the image under $T$ of the composed cobordism. Since the $\G$-structure is defined locally, the $\G$ structure of the inherited from the composed cobordism agrees with the $\G$-structure inherited from the concatenation of graphs. To see that this agrees with our notion of augmentations, we need only notice that the structured graph
\begin{equation*}
\xymatrix{
[n]\ar[r]^\phi & [1] & [m]\ar[l]_\psi
}
\end{equation*}
is a contraction of 
\begin{equation*}
\xymatrix{
[n]\ar[r]^\phi & [1] & [1]\ar[l]_{id} \ar[r]^{id} & [1] & [m]\ar[l]_\psi
}
\end{equation*}
Functorality in the non-stable case can be checked case-by-case. 

To see fully faithfulness in the non-stable case, we first notice that there is only one $\mathbb{G}$ structure on $D^2$ up to structured diffeomorphism. This means that, in the case of $D^2$ with a single marked point, we still get a single equivalence class of cobordisms. Since there is only a single equivalence class of structured augmented graphs with a single vertex and a single half-edge, fully faithfulness comes free. 

In the case of $D^2$ with two boundary marked points, we can identify the center vertex with one of the vertices on the edge. The remaining data amounts to an automorphism of $[1]$. Conversely, given any such automorphism $g$, taking the trivial $\mathbb{G}$-bundle over $D^2$ (with the same trivialization as over the first boundary interval), we can use the process outlined in  lemma \ref{lem:structrecon} to find a trivialization on the second boundary interval such that the $\G$-structure extracted from the resulting cobordism is precisely the structured graph corresponding to $g$. 

The remaining non-stable cases can be checked by cutting the surfaces into contractable pieces, and applying the two $D^2$ cases. 
\end{proof}

\begin{cor}
There is an equivalence of categories $\GOTFT_k\cong\CSFT_k$.
\end{cor}

\section{The Structure of CSFTs}

To classify CSFTs (and thus GOTFTs) we must first understand the structure with which a CSFT endows its target. For the remainder of this section, we will assume that we have a fixed CSFT corresponding to a balanced crossed simplicial group $\Delta\G$,
\begin{equation*}
Z:\Cor\to\Vect_k
\end{equation*}
and codify a set of properties this places on the target object. In the next section, we will show that this set of properties is sufficient to define a CSFT. 

The first point to notice, indicated by the use of the word `target' above, is that $Z$ defines an object 
\begin{equation*}
A:=Z(\scri)
\end{equation*}
The restriction of our functor $Z$ to the subcategory with morphisms in $P_{\Delta\G}$ yields a map of operads
\begin{equation*}
P_{\Delta\G}\to\text{End}(A)
\end{equation*}
where $\text{End}(A)$ is the endomorphism operad of $A$. This endows $A$ with the structure of an algebra over $P_{\Delta\G}$ as described in Proposition \ref{prop:assprop}. 

The additional structure we will explore will amount to understanding the restriction of $Z$ to a subcategory. To define this subcategory, however, we first need some additional notation.

\begin{defn}
A \textit{corolla} is a $\Delta\mathfrak{G}$-structured graph with only one vertex and no loops. A \textit{rose} is a $\Delta\mathfrak{G}$-structured graph with only one vertex. 
\end{defn}

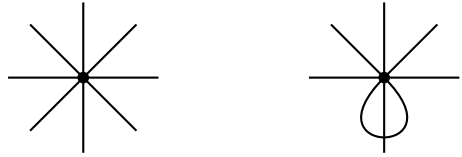
\begin{figure}[h]
\begin{tikzpicture}
\draw [fill=black] (1,1) circle [radius=0.07];
\draw [fill=black] (5,1) circle [radius=0.07];
\draw [thick] (1,1)--($ (1,1)+(0:1cm)$);
\draw [thick] (1,1)--($(1,1)+(45:1cm)$);
\draw [thick] (1,1) -- ($(1,1)+(90:1cm)$);
\draw [thick] (1,1) -- ($(1,1)+(135:1cm)$);
\draw [thick] (1,1) -- ($(1,1)+(180:1cm)$);
\draw [thick] (1,1) -- ($ (1,1)+(225:1cm)$);
\draw [thick] (1,1) -- ($ (1,1)+(270:1cm)$);
\draw [thick] (1,1) -- ($(1,1)+(315:1cm)$);
\draw [thick] (5,1)--($ (5,1)+(0:1cm)$);
\draw [thick] (5,1)--($(5,1)+(45:1cm)$);
\draw [thick] (5,1) -- ($(5,1)+(90:1cm)$);
\draw [thick] (5,1) -- ($(5,1)+(135:1cm)$);
\draw [thick] (5,1) -- ($(5,1)+(180:1cm)$);
\draw [thick] (5,1) -- ($ (5,1)+(270:1cm)$);
\draw [thick] (5,1) to [out=225,in=315,loop, distance=1.5cm] (5,1);
\end{tikzpicture}
\caption{A corolla (left) with 8 half edges and a rose with 8 half-edges and one loop.}
\end{figure}

We also introduce some terminology for discussing specific types of corollas (see figure \ref{fig:corfig}). For the sake of clarity, we will use gothic letters for morphisms in $\Cor$ and greek letters for morphisms in $\Vect_k$. 

\begin{figure}[htb]
\begin{tabular}{| m{2.5cm} | m{4cm} | m{2.5cm} |}
\hline
Name & 
Definition &
Picture \\ \hline
$n$-Trace &
A corolla with $n$ half-edges all labeled "in" under the augmentation. &
\begin{tikzpicture}
\path(0,0) node(c){} 
($ (0,0)+(45:1cm) $) node(f){1} 
($ (0,0)+(90:1cm) $) node(g){1}
($ (0,0)+(135:1cm) $) node(h){1};
\draw [thick] (c)--(f) (c)--(g) (c)--(h);
\draw [fill=black] (c) circle [radius=0.07];
\end{tikzpicture}\\
\hline
$n$-Cotrace &
A corolla with $n$ half-edges all labeled "out" under the augmentation. &

\begin{tikzpicture}
\path(0,0) node(c){} 
($ (0,0)+(-45:1cm) $) node(f){0} 
($ (0,0)+(-90:1cm) $) node(g){0}
($ (0,0)+(-135:1cm) $) node(h){0};
\draw [thick] (c)--(f) (c)--(g) (c)--(h);
\draw [fill=black] (c) circle [radius=0.07];
\end{tikzpicture}\\ \hline
Muliplication &
Elements of the operad $P_{\Delta\G}$ &
\begin{tikzpicture}
\path(0,0) node(c){} 
($ (0,0)+(45:1cm) $) node(f){1} 
($ (0,0)+(-90:1cm) $) node(g){0}
($ (0,0)+(135:1cm) $) node(h){1};
\draw [thick] (c)--(f) (c)--(g) (c)--(h);
\draw [fill=black] (c) circle [radius=0.07];
\end{tikzpicture}\\ \hline
Comultiplication &
Morphisms $\scri\to\scri^{\otimes k}$ in $\Cor$ represented by corollas &
\begin{tikzpicture}
\path(0,0) node(c){} 
($ (0,0)+(-45:1cm) $) node(f){0} 
($ (0,0)+(90:1cm) $) node(g){1}
($ (0,0)+(-135:1cm) $) node(h){0};
\draw [thick] (c)--(f) (c)--(g) (c)--(h);
\draw [fill=black] (c) circle [radius=0.07];
\end{tikzpicture}\\ \hline
\end{tabular}
\caption{Basic morphisms in $\Cor$}
\label{fig:corfig}
\end{figure}

\begin{prop}\label{prop:assprop}
For any balanced crossed simplicial group $\Delta\G$, there is a homomorphism 
\begin{equation*}
\chi_2:\G_0\to \Sigma_n\wr \G_0 
\end{equation*}
such that an algebra $A$ over the operad $P_{\Delta\G}$ is precisely a monoid equipped with an action of $\G_0$ such that the multiplication 
\begin{equation*}
m: A\otimes A\to A
\end{equation*}
is equivariant under $\chi_2$. In particular, $P_\Delta\G$ contains a copy of the associative operad $\mathcal{ASS}$.
\end{prop}

\begin{proof}
For any equivalence class in $P_{\Delta\G}(n)$, we can choose a unique standard representative such that the outgoing augmentation map is the unique augmentation morphism $\phi\in\Delta$ sending $0$ to $0$. Since the set of possible incoming augmentation maps forms a torsor under the operadic action of $\G_0$, we get that $P_{\Delta\G}(n)$ forms a torsor under
\begin{equation*}
H:=\underset{n \text{ times}}{\underbrace{\G_0\times\cdots\times\G_0}}\wr \Sigma_n
\end{equation*}  
Trivializing this torsor will give us a copy of $\Ass$ in $P_{\Delta\G}$. 

Following \cite{DK}, we consider the diagram in $\Delta$
\begin{equation*}
\xymatrix{
[1]\ar[drr]_{\{0,1\}} & [1]\ar[dr]^{\{1,2\}}  & & \cdots & [1]\ar[dll]^{\{n-1,n\}}\\
 & & [n] & & \\
 & & [1]\ar[u]_{\{0,n\}} & &
}
\end{equation*}
Applying the duality $D_\G$, we then obtain an augmented structured corolla, which we take as our trivialization of the torsor. To see that this choice respects composition, we simply compute that the pushout of a diagram in $\Delta$ of the form:
\begin{equation*}
\xymatrix{
[1]\ar[r]^{\{0,n\}}\ar[d]_{\{i-1,i\}} & [n]\\
[m] & 
}
\end{equation*} 
is given by the map $\gamma:[n]\to [n+m-1]$ with $\gamma(j)=i+j$ and the map  $\epsilon:[m]\to[n+m-1]$ given by 
\begin{equation*}
\epsilon(j)=\begin{cases}
j & j\leq i \\
j+n & j> i
\end{cases}
\end{equation*}
So that composing two such diagrams and taking a pushout gives us a diagram of the same form. Therefore, the trivializations given by the duality are closed under concatenation and contraction.

If we act on the outgoing edge of a multiplication $\mathfrak{m}$ by an element $h\in\G_0$, we can pull the group element back along the augmentation map by canonical factorization, so that the outgoing augmentation map of our new corolla is represented by $\phi\circ g$ for some $g\in \G_n$. Since two representatives of such multiplications are equivalent if and only if related by an automorphism of the central vertex, we can apply $g^{-1}$ to the central vertex, which gives us the standard representative for $\mathfrak{m}\circ h$. This changes the incoming augmentation maps by precomposing with $g^{-1}$, however, the original augmentation maps induce isomorphisms $\G_0\cong\Stab(k)$ for all $k\neq 0$ so that we can represent the new augmentation maps permuting the old ones according to $\lambda_n(g^{-1})$ and  postcomposing with elements of $\G_0$. That is, we can find an element of $H$ whose action on the incoming half-edges of $\mathfrak{m}$ gives the equivalence class of $\mathfrak{m}\circ h$. It is clear that this is compatible with the wreath product structure, and so for each $n$ we get a morphism 
\begin{equation*}
\chi_n:\G_0\to \Sigma_n\wr \G_0
\end{equation*}
under which the $n$-fold multiplication must be equivariant. However, it is easy to see that, under the composibility conditions for the copy of $\Ass$ in $P_{\Delta\G}$, it is sufficient to require that the multiplication $m_2:A\otimes A\to A$ be equivariant under $\chi_2$. 
\end{proof}

\begin{rem}
Working with an explicit copy of $\Ass$ in $P_{\Delta\G}$, one can compute more precisely what the condition specified by $\chi_2$ is. See the examples at the end of the section for more details.
\end{rem}

We can now introduce our subcategory of `generators' for $\Cor$. 

\begin{defn}
The category $\Gen$ is the subcategory of $\Cor$ generated under composition and disjoint union by the morphisms in $P_{\Delta\G}$, the unique 1-trace $\mathfrak{b}_1$ and the unique 1-cotrace $\mathfrak{p}_1$. It is clear that $\Gen$ will consist of all morphisms given by multiplications or traces.
\end{defn}

\begin{rem}
We made reference in the definition to \emph{the} 1-trace $\mathfrak{b}_1$. There is, in fact, only one such object. Since any incoming augmentation map $[0]\to[1]$ can be obtained from any other by applying an automorphism of $[0]$, all such morphisms represent the same equivalence class. In particular, this means that $\mathfrak{b}_1$ is invariant under the operadic action of $\G_0$. Similarly for the 1-cotrace.
\end{rem}

We now characterize functors $\Gen\to\Vect_k$. Taking the restriction $X:=Z\vert_\Gen$, it is obvious that $X$ will be completely determined by its values on $P_{\Delta\G}$ and $\mathfrak{b}_1$. Taking the copy of $\Ass$ in $P_{\Delta\G}$ with multiplications $id_A=\mathfrak{m_1},\mathfrak{m_2},\ldots$, we get traces 
\begin{equation*}
\mathfrak{b}_i=\mathfrak{b}_1\circ\mathfrak{m}_i
\end{equation*}

It is easy to see that, for any 2-trace, there is a uniquely defined 2-cotrace $\mathfrak{p}_2$ such that $\mathfrak{b}_2$ and $\mathfrak{p}_2$ compose to the identity. Graphically, figure \ref{fig:trfig} displays the non-degeneracy of $\mathfrak{b}_2$.
\begin{figure}[h]
\includegraphics[scale=1]{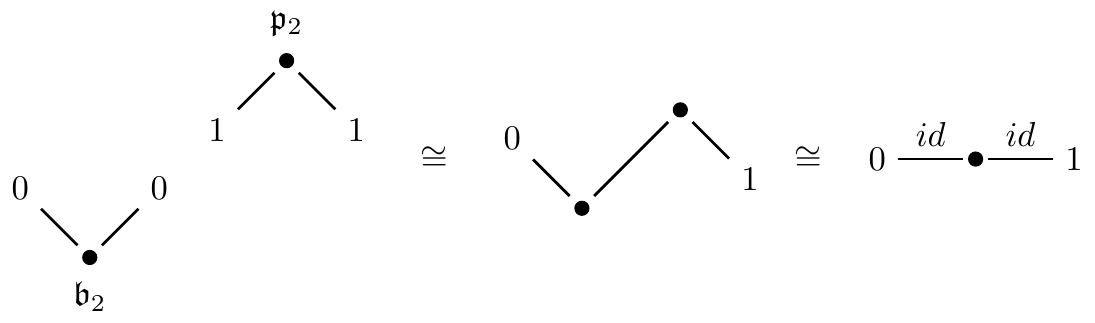}
\caption{Composing a 2-trace $\mathfrak{b}_2$ with the unique 2-cotrace displaying non-degeneracy of $Z(\mathfrak{b}_2)$.}
\label{fig:trfig} 
\end{figure}

\begin{lem}\label{lem:invhomo}
For any $n$-trace representative $\mathfrak{y}$ There is a homomorphism $\eta^\mathfrak{y}_n:\mathfrak{G}_n\to \mathfrak{G}_0\wr \Sigma_{n+1}$ such that the representative isomorphic to $\mathfrak{y}$ via $g$ is given by the operadic action of $\eta^\mathfrak{y}_n(g)$ on the incoming half-edges of $\mathfrak{y}$ 
\end{lem} 

\begin{proof}
We notice that, given such $g\in \G_n$, it induces a permutation of $\{0,1,\ldots, n\}=\Hom_\Delta ([0],[n])$ via the map $\lambda_n$. This permutation gives us a new labeling of the half-edges. 

The condition that we want, expressed in terms of the standard form of the original corolla, is that
\begin{equation*}
\phi_j \circ g = h_i\circ \phi_i 
\end{equation*} 
for a unique $h_i\in \G_1^0$. More clearly, if we let
\begin{equation*}
\Hom_{\Delta\G} ([n],[1])^i=\left\lbrace \psi\in \Hom_{\Delta\G} ([n],[1])\; \big\vert \; \psi(i)=1 \operatorname{ and } \psi^{-1}(1)=i \right\rbrace
\end{equation*}
then we want that the action of $\G_1^0$ on $\Hom_{\Delta\G} ([n],[1])^i$  by postcomposition is simply transitive. However, choosing a representative $\gamma\in \Hom_{\Delta\G} ([n],[1])^i$, we see that this is the same as saying that the subgroup $\gamma^\ast(\G_1^0)$ acts transitively by precomposition. 

Without loss of generality, we can reduce this to the case where $i=n$, since every morphism in $\Hom_{\Delta\G} ([n],[1])^i$ is given by a composition of a morphism in $\Hom_{\Delta\G} ([n],[1])^n$ with an element of $\G_n$. Reducing to this case, we see that the elements of $\G_n$ that act on $\Hom_{\Delta\G} ([n],[1])^n$ are precisely the members of $\operatorname{Stab}(n)\subset \G_n$. Since the action of $\G_n$ on $\Hom_{\Delta\G}([n],[0])$ is simply transitive, it suffices to show that $\G_0$ is isomorphic to $\operatorname{Stab}(n)$ via the homomorphism induced by pullback.

However, this is precisely the statement of lemma \ref{lem:augpullback}, so the proposition is proved.
\end{proof}
\begin{rem}\label{rem:tracehomo}
Given a trace $\mathfrak{y}$ with augmentation morphisms given by $\phi_i=\psi\circ g_i$, we can compute the form of the homomorphism $\eta_n^\mathfrak{y}$. Allowing $n$ to represent the morphism in $\Hom_\Delta ([0],[n])$ with target $n$, we can write:
\begin{eqnarray}\label{invform}
\eta_n^\mathfrak{y}:\G_n & \to & \G_0\wr \Sigma_{n+1}\\
 g & \mapsto & \left(n^\ast (g_0\circ g\circ g_{\sigma^{-1}(0)}^{-1}),\ldots,n^\ast (g_n\circ g\circ g_{\sigma^{-1}(n)}^{-1}), \sigma \right)
\end{eqnarray}
where $\sigma:=\lambda_n(g)$. 

We will, in particular, denote by $\eta_n$ the invariance condition defined by $\mathfrak{b}_n$
\end{rem}

We now can list the data we have extracted from $Z$.

\begin{defn}
A $\Delta\G$ Frobenius Algebra consists of the following data:
\begin{itemize}
\item A unital associative algebra $A$ equipped with an action of $\G_0$ such that the multiplication $\mu_2$ is equivariant under $\chi_2$.
\item A non-degenerate trace $\beta_1:A\to k$ such that
\begin{equation*}
\beta_n=\beta_1\circ\mu_n:A^{\otimes n}\to k
\end{equation*}
is invariant under $\eta_n$.
\end{itemize}
\end{defn}

\begin{defn}
The category $\Gfrob_k$ has objects $\Delta\G$ Frobenius Algebras over $k$. A morphism between two such algebras $(A,\beta_1)$ and $(B,\xi_1)$ is given by a $G_0$-equivariant algebra homomorphism 
\begin{equation*}
h: A\to B
\end{equation*}
such that 
\begin{equation*}
\beta_1(a,b)=\xi_1(h(a),h(b))
\end{equation*} 
\end{defn}

\begin{lem}\label{lem:cangen}
A symmetric monoidal functor 
\begin{equation*}
X:\Gen\to \Vect_k
\end{equation*}
can be reconstructed from its target $\Delta\G$ Frobenius Algebra.
\end{lem}

\begin{proof}
We know what $X$ must assign on any given morphism. Clearly the assignment is well functorial and consistent on $P_{\Delta\G}$, and, by definition, it is functorial on $\Gen$. It remains to check that it is well-defined on equivalence classes. However, this amounts to showing that the image of a morphism does not depend on representative, ie, showing that the traces are invariant under automorphisms of the central vertex. That is, we need to show that the traces are invariant under all the homomorphisms $\eta_n^\mathfrak{n}$. However, since any n-trace can be represented by a composition of $\mathfrak{b}_n$ and an element of $\Sigma_n\wr\G_0$, we see that the $\eta_n$ invariance conditions are sufficient to guarantee the other invariance conditions are satisfied. 
\end{proof}

\subsection{Examples and Computations}
\begin{exmp}\label{ex:first}
The simplest case is the cyclic case $\Delta\G=\Lambda$, which corresponds to $GL^+(2,\R)$ so that $\Cob=2\text{Cob}^o$. In this case, we see that $P_{\Lambda}=\mathcal{ASS}$, and that $\G_0$ is trivial. As a result, the action of an element of $g\in\G_n$ on an $n$-trace given in terms of the operadic action on half-edges is just the action of $\lambda_n(g)\in\Z/(n+1)\Z$. As a result, we can simplify the condition of lemma \ref{lem:invhomo} to require simply that traces be invariant under cyclic permutation of inputs, as a result, we see that a $\Delta\G$-frobenius algebra is precisely a frobenius algebra.
\end{exmp}

\begin{exmp}\label{ex:dih}
The next simple case is the dihedral case $\Delta\G=\Xi$. Here, the corresponding additional datum on surfaces of a reduction of the structure group is trivial, so that $\Cob$ is just the unoriented cobordism category.  Since we can find a copy of $\Lambda$ in $\Xi$, we can take the copy of $\mathcal{ASS}\subset P_{\Xi}$ given by $\Lambda$-structured trees. In this case, we can compute $\chi_2$, and we see that  an algebra over $P_{\Xi}$ is an algebra $A$ with an anti-automorphism $\ast$. More precisely, if we pull back the non-trivial element $f\in\G_1^0$ (the element which simply reverses orientation), we see that it pulls back to the reflection of the center circle fixing the outgoing marked point, so that it switches the inputs. Pushing out along the incoming augmentation maps, we see that it amounts to reversing orientation in each case, so that we get
\begin{equation*}
\chi_2(f)=(f,f;(1,2))
\end{equation*}
ie, that $f$ acts as an anti-automorphism of the algebra in question.
We then see that $\Xi$-frobenius algebras are precisely frobenius algebras $(A,\ast,\beta_1)$ with involution such that $\beta_1(a^\ast)=\beta_1(a)$.
\end{exmp}

\begin{exmp}\label{ex:fullex}
In the $N$-cyclic case $\Delta\G=\Lambda_N$, corresponding to $\text{Spin}_N(2)$-structured surfaces, we have to be a little more careful. To find a copy of $\Ass$, we need a particular characterization of $\Lambda_N$, given in \cite{DK} example 1.24. Let $C$ be the unit circle in $\C$, and let $C_n$ be $C$ equipped with $n+1$ marked points $\{0,1,\ldots,n\}$ included into $C$ via the map
\begin{equation*}
k\mapsto \exp\left(\frac{2\pi i k}{n+1}\right)
\end{equation*} 
Fixing an $N$-sheeted cover $\tilde{C}\to C$, we can then describe $\Lambda_N$ in the following way: Its objects are $\langle n\rangle$ for all $n$. A morphism $\langle m\rangle\to\langle n\rangle$ is given by a homotopy class of monotone maps $C_m\to C_n$ preserving the marked points together with a lift to $\tilde{C}$.

Using this, we can define elements $\tau_i^n\in\G_n$ which will allow us to choose a copy of $\Ass$ consisting of multiplications $\mathfrak{m}_i$. Let $t_i^n$ be the automorphism of $[n]$ in $\Lambda$ sending $0$ to $i$, represented as a homotopy class of monotone maps $C_n\to C_n$ preserving the marked points. Then there is a lift $\tau_i^n$ of $t_i^n$ to $\tilde{C}$ which is homotopy equivalent to the smallest positive rotation of $\tilde{C}_{Nn}$ covering $t_i$. If we define, for each $n$, a multiplication $\mathfrak{m}_n$ with $n$ incoming half-edges via the augmentation maps $(\psi_n\circ\tau_2^n,\ldots,\psi_n\circ\tau_n^n,\psi_\circ\tau_0^n,\phi_n)$, where $\psi$ is the map in $\Hom_\Delta([n],[1])$ sending $n$ to $1$ and everything else to $0$, and $\phi_n$ is the map in $\Hom_\Delta([n],[1])$ sending $0$ to $0$ and everything else to $1$, it is trivial to verify that $\{\mathfrak{m}_n\}$ forms a system of multiplications. Moreover, we can see that the traces $\mathfrak{b}_n:=\beta_1\circ\mathfrak{m}_n$ can be represented by the maps $\{\psi_{n-1}\circ\tau_{n-1}^i\}_{i=0}^{n-1}$. However, since by construction, for any map $i\in\Hom_\Delta([0],[n])$, we have $i^\ast(\tau_j^n)=id_[0]$, we see that the homomorphisms $\eta_n$ are precisely the homomorphisms $L_n$ from the proof of theorem 1.37 in \cite{DK}. 

We can also calculate $\chi_2$. If we pull back an element $f\in\G_1^0$, which can be represented by a rotation of $\tilde{C}$ by $2k$ markings, along $\phi_2$, we get a rotation by $3k$ markings of $\tilde{C}$. Pushing this out along the incoming augmentation maps, we again get a rotation of $\tilde{C}$ by $2k$ markings. That is, 
\begin{equation*}
\chi_2(f)=(f,f;id)
\end{equation*} 
or, more usefully: $ f\circ\mathfrak{m}=\mathfrak{m}\circ (f\sqcup f)$. Hence, the elements of $\G_0$ act on $A$ by automorphisms.

Since this is the case, our definition of a $\Lambda_N$-frobenius algebra simplifies to the one from \cite{DK}, and so we have the following characterization.

A $\Lambda_N$-frobenius algebra is a finite-dimensional unital associative algebra $A$ together with a linear function $\beta_1:A\to k$ such that 
\begin{itemize}
\item The form $\beta_2(a,b)=\beta_1(ab)$ is a (non necessarily symmetric) non-degenerate bilinear form on $A$.
\item The Nakayama automorphism $F$ of $\beta_2$ is an algebra automorphism of $A$ such that $F^N=id_A$
\end{itemize}
\end{exmp}

\begin{exmp}\label{ex:ndih}
In the $N$-dihedral case $\Delta\G=\Xi_N$, we can use a similar characterization to the one for $\Lambda_N$ (this time allowing both orientation preserving and reversing circle maps). With this characterization we immediately find a copy of $\Lambda_N$ in $\Xi_N$, and using the construction from example \ref{ex:fullex}, we can again reduce our notion of a $\Delta\G$-frobenius algebra to that of \cite{DK}. In this case, we find that a $\Xi_N$-frobenius algebra is a $\Lambda_N$-frobenius algebra equipped with a trace-preserving involution.
\end{exmp}

\begin{exmp}\label{ex:last}
In the paracyclic case $\Delta\G=\Lambda_\infty$, we can again use a characterization with circle maps $C_n\to C_m$, this time using a lift to a chosen universal cover $\R\to C$. The construction in example \ref{ex:fullex} generalizes, and we find that a $\Lambda_\infty$-frobenius algebra is a $\Lambda_N$-frobenius algebra $A$ in which we no longer require that $F^N=id_A$. 

As in example \ref{ex:ndih}, we can carry our argument over to the paradihedral case $\Delta\G=\Xi_\infty$. In this case, we find that a $\Xi_N$-frobenius algebra is a $\Lambda_\infty$-frobenius algebra with a trace-preserving involution. 
\end{exmp}

%\begin{exmp}
%As in the dihedral case, in the quaternionic case $\Delta\G=\nabla$, we can find a copy of $\Lambda$ in $\nabla$. In this case, $\Cob$ is the cobordism category with morphisms $\text{Pin}_2^-(2)$-structured surfaces. Algebras over $P_{\Delta\G}$ are precisely 
%\end{exmp}

\section{The Construction of CSFTs}

Now that we have a characterization of data arising from a CSFT, we show that these data are, in fact, sufficient to construct a CSFT. We now assume that we have a $\Delta\G$-Frobenius Algebra $A$, and attempt to construct a CSFT $Z$ from it. 

From lemma \ref{lem:cangen}, we know that $A$ gives rise to a functor:
\begin{equation*}
X: \Gen\to \Vect_k
\end{equation*}
We first claim that our `generator subcategory' does in fact generate $\Cor$. in some sense. 

First we note that for any 2-cotrace $\mathfrak{p}$, there is a unique 2-trace $\mathfrak{q}$ such that $\mathfrak{p}$ and $\mathfrak{q}$ compose to the identify as in figure \ref{fig:trfig}. As a result, the functor $X$ together with the non-degeneracy of $X(\mathfrak{q})$ uniquely determines an assignment on $\mathfrak{p}$. We will write $\mathfrak{p}_2$ for the 2-cotrace which composes with $\mathfrak{b}_1$ to the identify.

\begin{prop}
Every morphism in $\Cor$ can be expressed in terms of elements of $\Gen$ and a 2-cotrace $\mathfrak{p}_2$.
\end{prop}

The following lemmas will suffice to prove the proposition:

\begin{lem}
Any corolla $\mathfrak{u}$ that is not a 1-trace can be expressed as a concatenation of copies of $\mathfrak{p}_2$ with an element of $P_{\Delta\G}$. 
\end{lem}

\begin{proof}
If $\mathfrak{u}\in P_{\Delta\G}$, the lemma is trivial. If not, concatenate every outgoing half-edge except one with a copy of $\mathfrak{b}_2$. The resulting corolla $\mathfrak{u}^\prime$ represents an element in $P_{\Delta\G}$. However, since $\mathfrak{p}_2$ concatenates with $\mathfrak{b}_2$ to be the identity, concatenating every altered half-edge in $\mathfrak{u}^\prime$ with $\mathfrak{p}_2$ yields $\mathfrak{u}$ back.
\end{proof}

\begin{lem} \label{rosestruct}
Any rose $\mathfrak{k}$ can be written in terms of concatenations of a corolla and a 2-trace.
\end{lem}
\begin{proof}
We will perform the computation in the case for a single loop, which then generalizes by induction. 

Suppose $\mathfrak{k}$ is a rose with a single loop. Then we can write the incidence diagram of $\tilde{I}_\mathfrak{k}$ as 
\begin{equation*}
\mathfrak{u}\overset{f}{\underset{g}{\rightrightarrows}} M
\end{equation*} 
where  $\mathfrak{u}$ is a corolla and $M$ is a $\Delta\mathfrak{G}$-structured set with 2 elements. 

We can expand this diagram to the graph
\begin{equation*}
\xymatrix{
 &
 M &
 \\
 \mathfrak{u}\ar[ur]^f \ar[dr]_g &
 &
 M\ar[ul]_{id} \ar[dl]^{id}\\
 &
 M &
 {}
}
\end{equation*} 
Which admits an obvious functor of incidence categories to $\Sigma$. 

The lift of the pullback functor to $\mathcal{G}$ is then given by:

\begin{equation*}
\xymatrix{
\mathfrak{u}\ar@/^/[rr]^f \ar@/_/[rr]_g \ar[ddd]_{id_\mathfrak{u}}&
&
M\ar[ddl]_{id_M} \ar[ddddl]^{id_M} \ar[ddd]^{id_M}\\
&
&
\\
 &
 M &
 \\
 \mathfrak{u}\ar[ur]^f \ar[dr]_g &
 &
 M\ar[ul]_{id} \ar[dl]^{id}\\
 &
 M &
 {}
}
\end{equation*} 
And by choosing labelings of the structured sets $M$ at the centers of the edges, we get an augmented corolla $\mathfrak{u}$ which, when concatenated with
 $M\overset{id}{\leftarrow} M\overset{id}{\rightarrow} M$, yields $\mathfrak{k}$.
\end{proof}

Using these results, we can lay out an algorithm for assigning a value $Z(\mathfrak{k})$ to a given morphism $\mathfrak{k}$ in $\Cor$:
\begin{enumerate}
\item If $\mathfrak{k}$ is represented by a rose, write $\mathfrak{k}$ as $(\mathfrak{b}_2\sqcup\ldots\sqcup\mathfrak{b}_2)\circ\mathfrak{s}$, where $\mathfrak{s}$ is represented by a corolla.
\item Write the corolla $\mathfrak{s}$ as a composition of $\mathfrak{m}\in P_{\Delta\G}$ with 2-cotraces $\mathfrak{p}$. 
\item Assign as $Z(\mathfrak{k})$ the appropriate composition of the morphisms $Z(\mathfrak{m})$, $Z(\mathfrak{p})$, and $Z(\mathfrak{b}_2)$.
\end{enumerate}

We now need to show that this assignment is functorial and independent of the choices made.

Firstly, we notice that it does not matter which 2-traces and 2-cotraces we use in the algorithm, as any two will be related by an element of $P_{\Delta\G}(1)$, so that changing which (co)traces we use will simply involve inserting an element of $\G_0$ and its inverse into the computation. 

For ease of writing, we assign some notation:
\begin{eqnarray*}
\beta_i & := & Z(\mathfrak{b}_i)\\
\mu_i & := & Z(\mathfrak{m}_i)\\
\rho & := & Z(\mathfrak{p})
\end{eqnarray*}
where $\mathfrak{p}$ is the 2-cotrace which composes with $\mathfrak{b}_2$ to the identity.

\begin{lem}
The assignment $Z(\mathfrak{k})$ for $\mathfrak{k}$ represented by a corolla does not depend on the choices made. 
\end{lem}

\begin{proof}
The only choice left to consider is which outgoing half-edges we compose with copies of $\mathfrak{b}_2$ to get an element of $P_{\Delta\G}$. Suppose that  We have two elements $\mathfrak{m}$ and $\mathfrak{n}$ obtained by this proceedure from $\mathfrak{k}$, omitting different outgoing half-edges. As in the diagram:
\begin{center}
\includegraphics[scale=1]{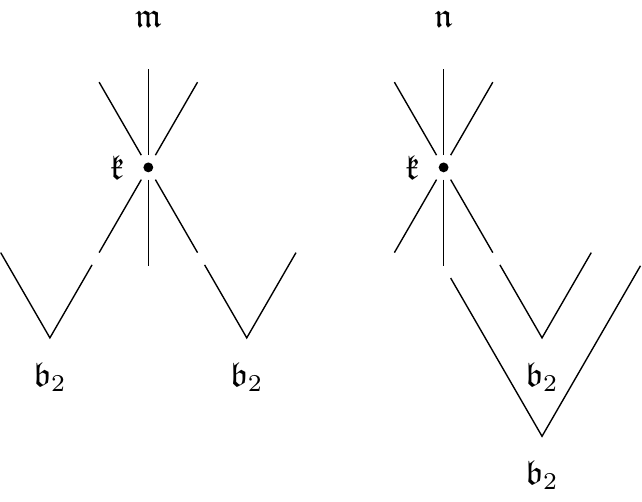}
\end{center}
 Then we have that 
\begin{equation*}
\mathfrak{m}=\mathfrak{b}_2\circ\mathfrak{n}\circ\mathfrak{p}
\end{equation*}
composed as in the diagram:
\begin{center}
\includegraphics[scale=1]{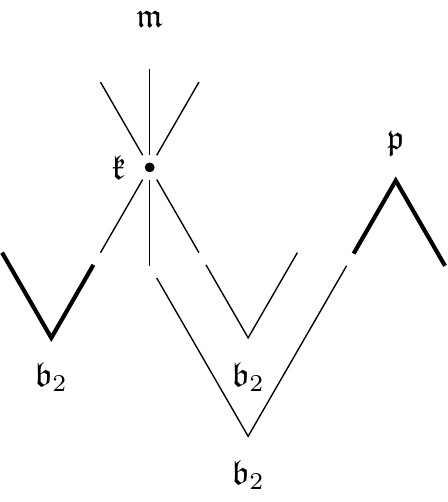}
\end{center}
However, since we already have a compatible assignment on $\Gen$ and $\mathfrak{p}$, we know that 
\begin{equation*}
Z(\mathfrak{m})=\beta_2\circ Z(\mathfrak{n})\circ \rho
\end{equation*}
And therefore, both choices yield the same value for $Z(\mathfrak{k})$.
\end{proof}

A comparable lemma for roses is unnecessary, since the only choice being made   is the identification of the structured set $M$ from lemma 6.3 with $[1]$. This choice, as mentioned above, amounts to inserting a group element and it's inverse into the computation, and so does not affect the value assigned to a given rose.

Before we continue, we will fix some notation to try and ease the writing. For two augmented $\G$-structured graphs $\mathfrak{u}_1$ and $\mathfrak{u}_2$ with a linear order of the incoming/outgoing half-edges we define the symbol 
\begin{equation*}
\mathfrak{u}_1 \circ_{m,n} \mathfrak{u}_2
\end{equation*}
to be the augmented $\G$-structured graph defined by concatenating the $m^{th}$ incoming half-edge of $\mathfrak{u}_1$ with $n^{th}$ outgoing half-edge of $\mathfrak{u}_2$ 

In addition, we will denote by 
\begin{equation*}
\left(\mathfrak{u}_1 \circ_{m,n} \mathfrak{u}_2 \right)\vert_{m,n}
\end{equation*}
the contraction of this concatenation along the newly created interior edge.

We will also use this notation for morphisms in $\Vect_k$, eg, for $\beta:A^{\otimes n} \to A^{\otimes m}$ and $\gamma:A^{\otimes k}\to A^{\otimes \ell}$ we will write 
\begin{equation*}
\beta \circ_{p,q} \gamma
\end{equation*}
to mean the composition of the $p$th input of $\beta$ with the $q$th output of $\gamma$. It is worth noting here that we are implicitly inserting copies of the identity at the other inputs/outputs (and we are also doing so for morphisms in $\Cor$) but we suppress this for ease of writing.

Now that we have a consistent definition of $Z$ on every morphism in $\Cor$, we need to check that this definition is functorial.

\begin{lem}[Functoriality for Single Compositions]\label{lem:corollacontraction}
Let $Z$ be defined as above. Given 2 corollas $\mathfrak{u}_1$ and $\mathfrak{u}_2$, then
\begin{eqnarray*}
Z(\mathfrak{u}_1) \circ_{m,n} Z(\mathfrak{u}_2)
 & = & Z(\left(\mathfrak{u}_1 \circ_{m,n} \mathfrak{u}_2\right)\vert_{m,n})
\end{eqnarray*}
\end{lem}
\begin{proof}
Suppose we have two corollas $\mathfrak{u}_1$ and $\mathfrak{u}_2$. 
Compose them as follows:
\begin{center}
 \includegraphics[scale=1]{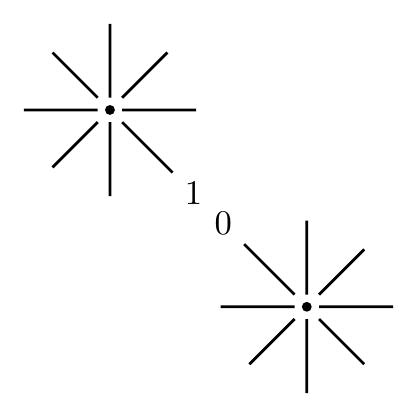}
\end{center}
and call their contraction $\mathfrak{f}$.

Let $\mathfrak{u}_i^\prime$ be the corolla given by composing $\mathfrak{u}_i$ with multiple copies of $\mathfrak{b}_2$  (doing this in 
such a way that the half-edges along which the $\mathfrak{u}_i$ are composed is left unchanged).  We then get limit diagrams:

\begin{equation*}
 \xymatrix @R=0.1pc{
  &
  &
  \mathfrak{u}_i^\prime \ar@[red][dddddr]\ar@[red][ddddrr]\ar@[red][dddddddr]\ar@[red][ddddddddrr]\ar@[red][dddddd]
  \ar@[red][dddddl]\ar@[red][ddddll]\ar@[red][dddddddl]\ar@[red][ddddddddll]&
  &
  \\
  & & & & \\
  & & & & \\
  & & & & \\
  \mathfrak{b}_2\ar[dr] & & & & \mathfrak{b}_2\ar[dl]\\
   & [1] & & [1] & \\ 
   \vdots & & \mathfrak{u}_i\ar[ur]\ar[ul]\ar[dr]\ar[dl] & & \vdots\\
     & [1] & & [1] & \\ 
     \mathfrak{b}_2\ar[ur] & & & & \mathfrak{b}_2\ar[ul]\\
 }
\end{equation*}

Writing the contraction of $\mathfrak{u}_1^\prime$ and $\mathfrak{u}_2^\prime$ as $\mathfrak{k}$, we also get limit diagrams:

\begin{equation*}
 \xymatrix{
 & 
 \mathfrak{k}\ar[d]\ar[dr]\ar[dl] & 
 \\
 \mathfrak{u}_1^\prime\ar[r] &
 [1] &
 \mathfrak{u}_2^\prime \ar[l]\\
 }
\end{equation*}
and
\begin{equation*}
 \xymatrix{
 & 
 \mathfrak{f}\ar[d]\ar[dr]\ar[dl] & 
 \\
 \mathfrak{u}_1\ar[r] &
 [1] &
 \mathfrak{u}_2 \ar[l]\\
 }
\end{equation*}

But, composing the morphisms in the limit diagram of $\mathfrak{k}$ with the morphisms $\mathfrak{u}_i^\prime\to\mathfrak{u}_i$, we can 
display $\mathfrak{k}$ as a cone over  $\mathfrak{u}_1 \to[1] \leftarrow \mathfrak{u}_2$. This means that, since $\mathfrak{f}$ is a limit, we get a 
unique map $\mathfrak{k}\to\mathfrak{f}$ such that TFDC:
\begin{equation*}
 \xymatrix{
  & 
  \mathfrak{k}\ar[d] \ar@[blue]@/^/[dd]\ar@[blue][ddr]\ar@[blue][ddl]&
  \\
  &
  \mathfrak{f}\ar[d]\ar[dr]\ar[dl] &
  \\
  \mathfrak{u}_1\ar[r] & 
  [1] &
  \mathfrak{u}_2\ar[l]
 }
\end{equation*}

Note that the cone in question is \textit{not} a contraction, since it does not respect the augmentation on the external 
half-edges. However, since the diagram above commutes, we can extend it to 
\begin{equation*}
 \xymatrix @R=0.1pc{
  &
  &
  &
  \mathfrak{k} \ar@[red][dddddrr]\ar@[red][ddddrrr]\ar@[red][dddddddrr]\ar@[red][ddddddddrrr]\ar@[red][dddddd]\ar@[red][ddddddr]\ar@[red][ddddddl]
  \ar@[red][dddddll]\ar@[red][ddddlll]\ar@[red][dddddddll]\ar@[red][ddddddddlll]&
 &
 &
  \\
  & & & & & & \\
  & & & & & & \\
  & & & & & & \\
  \mathfrak{b}_2\ar[dr] & & & & & & \mathfrak{b}_2\ar[dl]\\
   & [1] & & & & [1] & \\ 
   \vdots & & \mathfrak{u}_1\ar[ul]\ar[dl]\ar[r] & [1]&\mathfrak{u}_2\ar[ur]\ar[dr]\ar[l] & & \vdots\\
     & [1] & & & & [1] & \\ 
     \mathfrak{b}_2\ar[ur] & & & & & & \mathfrak{b}_2\ar[ul]\\
 }
\end{equation*}
by taking the compositions of $\mathfrak{k}\to \mathfrak{u}_i^\prime$ with $\mathfrak{u}_i^\prime\to \mathfrak{b}_2$ and $\mathfrak{u}_i^\prime \to [1]$ respectively.

However, we can also extend the limit diagram of $\mathfrak{f}$ to:
\begin{equation*}
 \xymatrix @R=0.1pc{
 \beta_2\ar[dr]\ar@[blue]@/_/[dddddddd]  & & & & & & \beta_2\ar[dl] \ar@[blue]@/^/[dddddddd] \\
   & [1]\ar@[blue]@/_/[dddddddd]  & & & & [1]\ar@[blue]@/^/[dddddddd]  & \\ 
   \vdots & &  & \mathfrak{f} \ar[ull]\ar[urr] \ar[dll]\ar[drr]\ar@[blue][dddddddd] \ar@[blue][ddddddddr] \ar@[blue][ddddddddl]  & & & \vdots\\
     & [1] \ar@[blue][dddddddd] & & & & [1]\ar@[blue][dddddddd]  & \\ 
     \mathfrak{b}_2\ar[ur]\ar@[blue][dddddddd]  & & & & & & \mathfrak{b}_2\ar[ul]\ar@[blue][dddddddd] \\
  & & & & & & \\
  & & & & & & \\
  & & & & & & \\
  \beta_2\ar[dr] & & & & & & \beta_2\ar[dl]\\
   & [1] & & & & [1] & \\ 
   \vdots & & \mathfrak{u}_1\ar[ul]\ar[dl]\ar[r] & [1]&\mathfrak{u}_2\ar[ur]\ar[dr]\ar[l] & & \vdots\\
     & [1] & & & & [1] & \\ 
     \mathfrak{b}_2\ar[ur] & & & & & & \mathfrak{b}_2\ar[ul]\\
 }
\end{equation*}
Here, the maps $\mathfrak{b}_2\to\mathfrak{b}_2$ and $[1]\to[1]$ are identities (we can extend in this way because the augmentation maps for $\mathfrak{f}$ 
are defined by those for $\mathfrak{u}_1$ and $\mathfrak{u}_2$). 

Since this is the case we can pull back the red arrows from the previous diagram along these identities to get a diagram

\begin{equation*}
  \xymatrix @R=0.1pc{
  &
  &
 \mathfrak{k} \ar@[violet][dddddr]\ar@[violet][ddddrr]\ar@[violet][dddddddr]\ar@[violet][ddddddddrr]\ar@[violet][dddddd]
  \ar@[violet][dddddl]\ar@[violet][ddddll]\ar@[violet][dddddddl]\ar@[violet][ddddddddll]&
  &
  \\
  & & & & \\
  & & & & \\
  & & & & \\
  \mathfrak{b}_2\ar[dr] & & & & \mathfrak{b}_2\ar[dl]\\
   & [1] & & [1] & \\ 
   \vdots & & \mathfrak{f} \ar[ur]\ar[ul]\ar[dr]\ar[dl] & & \vdots\\
     & [1] & & [1] & \\ 
     \mathfrak{b}_2\ar[ur] & & & & \mathfrak{b}_2\ar[ul]\\
 }
\end{equation*}

Displaying $\mathfrak{k}$ as a cone over the concatenation of $\mathfrak{f}$ with copies of $\mathfrak{b}$. Since the augmentation maps agree now,
we get that this is an equivalence of augmented structured graphs. That is, $\mathfrak{k}$ is precisely the contraction of 
the concatenation of $\mathfrak{f}$ with copies of $\mathfrak{b}_2$.
\end{proof}

\begin{rem}
A similar argument shows functorality for a pair of roses composed at a single half-edge.
\end{rem}

\begin{lem}[Functoriality for Multiple Compositions]\label{lem:rosecontraction}
Let $Z$ be defined as above. Given 2 corollas $\mathfrak{u}_1$ and $\mathfrak{u}_2$, then
\begin{eqnarray*}
Z(\mathfrak{u}_1) \circ_{m,n}\circ_{k,\ell} Z(\mathfrak{u}_2)
 & = & Z\left(\left(\mathfrak{u}_1 \circ_{m,n}\circ_{k,\ell} \mathfrak{u}_2\right)\vert_{m,n}\right) \\
 & = & Z\left(\left(\mathfrak{u}_1 \circ_{m,n}\circ_{k,\ell} \mathfrak{u}_2\right)\vert_{k,l}\right)
\end{eqnarray*}
\end{lem}

\begin{proof}
We have a diagram of the form
\begin{equation*}
\xymatrix{
 & \mathfrak{u}_2\ar[dr]\ar[dl] & \\
 [1] & & [1]\\
 & \mathfrak{u}_1\ar[ur]\ar[ul] & 
}
\end{equation*}
Which, contracted along the first new edge, gives us a new diagram of the form
\begin{equation*}
\xymatrix{
\mathfrak{r}\ar@/^/[r]\ar@/_/[r] & M
}
\end{equation*}
To compute the value of $Z$ on this corolla (ie, to compute $Z\left(\left(\mathfrak{u}_1 \circ_{m,n}\circ_{k,\ell} \mathfrak{u}_2\right)\vert_{m,n}\right)$), we  expand this to a diagram
\begin{equation*}
\xymatrix{
 & M & \\
\mathfrak{r}\ar[ur]\ar[dr]  & & M \ar[ul]\ar[dl]\\
 & M &
}
\end{equation*}
We then identify (WLOG) $M\leftarrow M\to M$ with $\mathfrak{p}_2$, and compute the value of $Z$ on $\mathfrak{r}$. However, by the previous lemma, $Z(\mathfrak{r})$ can be computed by computing the value on the diagram 
\begin{equation*}
\xymatrix{
  \mathfrak{u}_2\ar[dr]\ar[d] & & & \\
 [1] & [1] & & [1] \\
 &  & \mathfrak{b}_2\ar[ur]\ar[ul] &\\
  \mathfrak{u}_1\ar[uu]\ar[rrr] & & & [1]
}
\end{equation*}
This tells us that $Z\left(\left(\mathfrak{u}_1 \circ_{m,n}\circ_{k,\ell} \mathfrak{u}_2\right)\vert_{m,n}\right)$ can be computed by apply $Z$ to  pieces of 
\begin{equation*}
\xymatrix{
  \mathfrak{u}_2\ar[dr]\ar[d] & & & \mathfrak{p}_2\ar[d]\ar@/^2pc/[ddd] \\
 [1] & [1] & & [1] \\
 &  & \mathfrak{b}_2\ar[ur]\ar[ul] &\\
  \mathfrak{u}_1\ar[uu]\ar[rrr] & & & [1]
}
\end{equation*}
However, we know that $\beta_2=Z(\mathfrak{b}_2)$ and $\rho=Z(\mathfrak{p}_2)$ compose to the identity, so apply $Z$ to this diagram also computes 
$Z(\mathfrak{u}_1) \circ_{m,n}\circ_{k,\ell} Z(\mathfrak{u}_2)$, proving the lemma. 
\end{proof}

\begin{rem}
The same proof works for pairs of roses composed along more than two half-edges. In this case, we simply get several copies of $\mathfrak{b}_2$ and $\mathfrak{p}_2$.
\end{rem}

What we have now shown is the following:
\begin{prop}\label{prop:CSFTclass}
For a balanced crossed simplicial group $\Delta\G$, given a $\Delta\G$-frobenius algebra $A$ over $k$, there is a CSFT $Z$ with $Z(\scri)=A$.
\end{prop}

As a corollary, we get our main result:

\begin{thm}\label{thm:CSFTfin}
For a balanced crossed simplicial group $\Delta\G$, there is an equivalence of categories
\begin{equation*}
\CSFT_k\cong \Gfrob_k
\end{equation*}
\end{thm}

And, applying \ref{thm:cobequiv}, we get as a further corollary

\begin{cor}\label{cor:fincor}
For a planar lie group $G$ corresponding to a crossed simplicial group $\Delta\G$, there is an equivalence of categories
\begin{equation*}
\GOTFT_k\cong \Gfrob_k
\end{equation*}
\end{cor}

Making reference to the examples \ref{ex:first} and \ref{ex:dih}, we retrieve the following results:
\begin{exmps}
Oriented open topological field theories over $k$ are equivalent to frobenius algebras over $k$. (folklore)

Unoriented open topological field theories over $k$ are equivalent to frobenius algebras over $k$ equipped with a trace-preserving anti-automorphism of order 2. (See\cite{B})
\end{exmps}

\subsection{Equivariant TFT's}

Not only does the framework provided by balanced crossed simplicial groups provide a classification of field theories on $\Cob$ for a planar Lie group $G$, it also provides a classification of so-called \emph{equivariant} field theories.  

Recall from example \ref{ex:BH} that, given a finite group $H$ and a planar crossed simplicial group $\Delta\G$ we can form a new balanced crossed simplicial group $\Delta\mathfrak{GH}=\Delta\G\times BH$. This new crossed simplicial group admits a forgetful functor
\begin{equation*}
F_H:\Delta\mathfrak{GH}\to\Delta\G
\end{equation*}
given by projection onto the first component. It is also immediate from the definitions that the functor
\begin{equation*}
\lambda_\mathfrak{GH}:\Delta\mathfrak{GH}\to \mathbf{FSet}
\end{equation*}
factors as:
\begin{equation*}
\xymatrix{
\Delta\mathfrak{GH}\ar[dd]_{\lambda_\mathfrak{GH}}\ar[dr]^{F_H} & \\
& \Delta\G \ar[dl]_{\lambda_\G}\\
\mathbf{FSet} & 
}
\end{equation*}

As a consequence, we have induced functors amongst the structured set categories:
\begin{equation*}
\xymatrix{
\cgh\ar[dd]_{\lambda_\cgh}\ar[dr]^{F_H} & \\
& \cG \ar[dl]_{\lambda_\cG}\\
\mathbf{FSet} & 
}
\end{equation*}

As a result of this factorization, we can interpret the notion of $\Delta\mathfrak{GH}$-structured graph in a more useful way. A $\Delta\GH$-structured graph is a lift of incidence diagrams.
\begin{equation*}
\xymatrix{
 & \cgh \ar[d]^{\lambda_\cgh} \\
 I(\Gamma)\ar[r]_{I_\Gamma}\ar[ur]^{\tilde{I}_\Gamma} & \mathbf{FSet}
}
\end{equation*}
which can then be viewed as a pair of lifts:
\begin{equation*}
\xymatrix{
 & \cgh \ar[d]^{F_H} \\
  & \cG\ar[d]^{\lambda_\cG} \\
 I(\Gamma)\ar[r]_{I_\Gamma}\ar[ur]_{\tilde{I}_\Gamma} \ar[uur]^{\hat{I}_\Gamma} & \mathbf{FSet}
}
\end{equation*}
That is, a $\Delta\GH$-structured graph is a $\Delta\G$-structured graph together with additional data relating to the group $H$.  
This additional data is, more precisely, for every edge or vertex, an $H$-torsor $\mathscr{O}_H$, together with maps of torsors corresponding to half-edges. Since $\lambda_\GH$ restricted to morphisms in $H$ is trivial, there are no other data involved in the lift. To make this graph augmented amounts to choosing a trivialization of the torsors associated to the external half-edges.

Now, suppose that $S$ is a $\mathbb{G}$-structured surface, for a planar lie group $G$ corresponding to a planar crossed simplicial group $\Delta\G$. Let $\Gamma$ be the augmented spanning graph associated to $S$ as in the proof of theorem \ref{thm:cobequiv}, with $\Delta\G$-structure $\tilde{I}_\Gamma$.

As it turns out, there is a connection between principal $H$-bungles on $S$ and $\Delta\GH$-structures on $\Gamma$. 

\begin{lem}\label{lem:principalbundleequiv}
Let $S$ be a morphism in $\Cob$, corresponding to an equivalence class $\mathfrak{m}$ of  $\Delta\G$ structured graphs via the equivalence of theorem \ref{thm:cobequiv}. Then a $\Delta\GH$-structure on $\mathfrak{m}$ is equivalent to an $H$-principal $E$ bundle on $S$. 
\end{lem} 

\begin{proof}
 We label the vertices and edges (including augmentations), by a finite indexing set $I$. It is easy to see from the ribbon/m\"obius graph thickening proceedures that we can find contractible neighborhoods $\{U_i\}$ covering $S$, each containing precisely one of the vertices of the embedded graph, as seen in figure \ref{fig:prinbun}. 
 
 \begin{figure}[h]
  \includegraphics[scale=1]{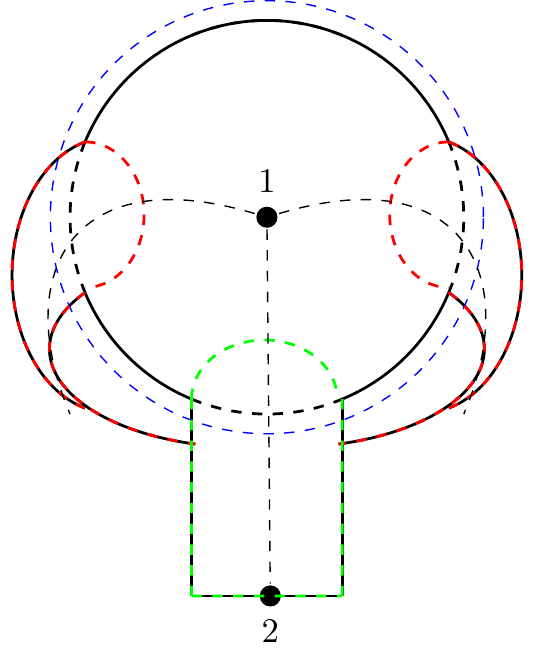}
  \caption{A thickened ribbon graph covered by contractible sets.}
  \label{fig:prinbun}
 \end{figure}

To each point $i$ there is an associated $H$-torsor $\mathscr{O}_i$, so we can give the bundle locally as $E\vert_{U_i}=U_i\times \mathscr{O}_i$. The maps of torsors give transition functions, which vacuously satisfy the cocycle condition, and the trivializations of the augmented half-edges yield trivializations on embedded boundary components. This construction is obviously invariant under automorphisms of the vertices. Contraction simply means taking a union of two of the $U_i$'s (that do not form a closed loop, ie do not admit a retraction to $S^1$), and identifying the torsors over them via the appropriate transition function. Therefore, the construction is well defined on equivalence classes of graphs.

Now, supposing we have the representative $\Gamma$ embedded in $S$, with a principal $H$-bundle $\pi:E\to S$, we can take the $H$-torsor for the vertex/edge $i$,  $\mathscr(O)_i$, to be $\pi^{-1}(i)$. The transition functions give maps between the torsors, recovering the required structure. Now, the $\Delta\GH$-structured graph $\Gamma$ uniquely determines a morphism in $\GHbord$ (equivalence class).

It is clear that these constructions are inverse to one another, and so the lemma is proved.
\end{proof}

\begin{defn}
 For $H$ a finite group, we define a symmetric monoidal category $\Hcob$, whose morphisms are morphisms in $\Cob$ equipped with a principal $H$-bundle, and trivializations thereof on the marked boundary components. Gluing of morphisms takes place in the obvious way. 
\end{defn}

\begin{rem}
It is worth noting that the construction in lemma \ref{lem:principalbundleequiv} also clearly sends the concatenation of graphs to the composition of bordisms, so that, in fact, we have an equivalence of categories:
\begin{equation*}
\Hcob \cong \GHbord
\end{equation*}
\end{rem}

So we have proved, as a corolary to lemma \ref{lem:principalbundleequiv} and theorem \ref{thm:CSFTfin}, the following:

\begin{cor}
 Topological field theories
 \begin{equation*}
 Z:\Hcob \to \Vect_k
 \end{equation*}
are equivalent to $\Delta\GH$-frobenius algebras over $k$.
\end{cor}

\begin{rem}
 We can, in fact, give a more complete characterization of a $\Delta\GH$-frobenius algebra. It will be $\Delta\G$-frobenius algebra $A$, along with an action of $H$ on $A$ by algebra automorphisms such that the 1-trace is invariant under the action of $H$. 
\end{rem}


\newpage

\begin{thebibliography}{9}

\bibitem[B]{B}
C. Braun.
\textit{Moduli Spaces of Klein Surfaces and Related Operads}
arXiv:1003.5903

\bibitem[D]{D}
T. Dyckerhoff.
\textit{$\mathbb{A}^1$-homotopy invariants of topological Fukaya categories of surfaces}
arXiv:1505.06941

\bibitem[DK]{DK}
T. Dyckerhoff and M. Kapranov.
\textit{Crossed Simplicial Groups and Structured Surfaces.}
arXiv:1403.5799

\bibitem[DK2]{DK2}
T. Dyckerhoff and M. Kapranov.
\textit{Triangulated surfaces in triangulated categories.}
arXiv:1306.2545

\bibitem[FL]{FL} 
Z. Fiedorowicz and J. Loday.
\textit{ Crossed Simplicial Groups}


\bibitem[K]{K} J. Kock. \textit{Frobenius Algebras and Topological Quantum Field Theory} London Math, Soc. Lecture Series 59. Cambridge Univ. Press, 2004.

\bibitem[Kr]{Kr}
R. Krasauskas. \textit{Skew-simplicial Groups}. Translated from Litovskii Matematichenskii Sbornik (Lietuvos Matematikos Rinkinys),
Vol. 27, No. 1, pp. 89-99, January-March, 1987

\bibitem[L]{L}
J. Lurie.
\textit{On the Classification of Topological Field Theories.}
arXiv:0905.0465

\bibitem[NR]{NR}
S. Novak and I. Runkel.
\textit{State sum construction of two-dimensional topological quantum field theories on spin surfaces.}
arXiv:1402.2839 


%\bibitem[Ka]{Ka}
%L.H. Kauffman. 
%\textit{Hopf algebras and invariants of 3-manifolds}. Journal of Pure and %Applied Algebra. Vol 100 pp. 73-92.  1995.

%\bibitem[CGW]{CGW} 
%M. Cohen, S. Gelaki, S. Westreich. 
%\textit{Hopf Algebras}. Handbook of Algebra Volume 4. pp. 175-198.

%\bibitem[M]{M}
%S. Montgomery. 
%\textit{Representation Theory of Finite Semisimple Hopf Algebras}. Algebra--Representation Theory, pp. 189-218. Springer Science \& Business Media. 2001.



\end{thebibliography}
\end{document}